\documentclass[11pt,leqno]{article}

\usepackage[utf8]{inputenc}

\usepackage[affil-it]{authblk}
\usepackage{combelow}
\usepackage{titlefoot}
\usepackage{titlesec}
	\titleformat{\section}[block]{\Large\bfseries\filcenter}{\thesection}{1em}{}
\usepackage{setspace}

\usepackage[final]{optional}
\usepackage[british,american]{babel}
\usepackage[margin=3cm]{geometry}
\usepackage{lipsum}
\usepackage[leqno]{amsmath}
\usepackage{mathrsfs}
\usepackage{indentfirst}

\let\oldbibliography\thebibliography
\renewcommand{\thebibliography}[1]{%
  \oldbibliography{#1}%
  \setlength{\itemsep}{-.5pt}%
}

\usepackage{amsfonts, amsmath, wasysym, caption}
\usepackage{amssymb,amsthm, paralist,enumerate}
\usepackage{mathtools}
\usepackage{esint}
\usepackage{bbm}

\makeatletter
\newcommand{\myitem}[1]{%
\item[#1]\protected@edef\@currentlabel{#1}%
}
\makeatother

\usepackage{
latexsym,
}

\usepackage[usenames]{color}
\usepackage{tikz}
\usetikzlibrary{decorations.pathreplacing}
\usepackage{url}

\definecolor{darkgreen}{rgb}{0,0.5,0}
\definecolor{darkred}{rgb}{0.7,0,0}

\usepackage{hyperref}

\usepackage{esint}
\usepackage[normalem]{ulem}

\usepackage{enumitem}

\setenumerate{label={\rm (\roman{*})},itemsep=0em}

\usepackage{pifont}
\newcommand{\cmark}{\ding{51}}%
\newcommand{\xmark}{\ding{55}}%

\theoremstyle{plain}
\newtheorem{bigthm}{Theorem}

\newtheorem{bigcor}[bigthm]{Corollary}

\theoremstyle{plain}
\newtheorem{lemma}[equation]{Lemma}
\newtheorem{thm}[equation]{Theorem}
\newtheorem{prop}[equation]{Proposition}
\newtheorem{cor}[equation]{Corollary}
\theoremstyle{definition}
\newtheorem{defn}[equation]{Definition}

\newtheorem{ex}[equation]{Example}
\newtheorem{rmk}[equation]{Remark}

\expandafter\let\expandafter\oldproof\csname\string\proof\endcsname
\let\oldendproof\endproof
\renewenvironment{proof}[1][\proofname]{%
  \oldproof[\upshape \bfseries #1:]%
}{\oldendproof}

\numberwithin{equation}{section}

\newcommand{\al}{\alpha}
\newcommand{\be}{\beta}

\newcommand{\De}{\Delta}
\newcommand{\om}{\omega}
\newcommand{\Om}{\Omega}

\newcommand{\la}{\lambda}

\renewcommand{\th}{\theta}

\newcommand{\D}{\textup{D}}
\newcommand{\R}{\ensuremath{{\mathbb R}}}
\newcommand{\N}{\ensuremath{{\mathbb N}}}

\DeclareMathOperator{\curl}{curl}

\newcommand{\locc}{\operatorname{loc}}
\newcommand{\lebe}{L}
\newcommand{\sobo}{W}
\newcommand{\hold}{C}
\newcommand{\dif}{\operatorname{d}\!}
\newcommand{\cala}{\mathcal{A}}
\newcommand{\B}{\mathcal{B}}

\makeatletter
\newcommand*\bigcdot{\mathpalette\bigcdot@{.5}}
\newcommand*\bigcdot@[2]{\mathbin{\vcenter{\hbox{\scalebox{#2}{$\m@th#1\bullet$}}}}}
\makeatother

\newcommand{\weakto}{\rightharpoonup}
\newcommand{\wstar}{\overset{\ast}{\rightharpoonup}}

\let\div\relax
\DeclareMathOperator{\div}{div}

\newcommand{\brmk}{\begin{rmk}}
\newcommand{\ermk}{\end{rmk}}
\newcommand{\partref}[1]{\hbox{(\csname @roman\endcsname{\ref{#1}})}}

\newcommand{\tp}{\textup}
\newcommand{\hs}{\hspace{0.5cm}}

\newcommand{\beq}{\begin{equation}}
\newcommand{\eeq}{\end{equation}}
\newcommand{\beqs}{\begin{equation*}}
\newcommand{\eeqs}{\end{equation*}}
\newcommand{\beqa}{\begin{equation}\begin{aligned}}
\newcommand{\eeqa}{\end{aligned}\end{equation}}
\newcommand{\beqas}{\begin{equation*}\begin{aligned}}
\newcommand{\eeqas}{\end{aligned}\end{equation*}}

\newcommand{\half}{\frac{1}{2}}
\newcommand{\eps}{\varepsilon}
\newcommand{\supp}{\text{supp}}

\begin{document}

 \title{\LARGE \textbf{Compensated Compactness:\\
 continuity in optimal weak topologies}}

\author[1]{{\Large Andr\'e Guerra}}
\author[2]{{\Large  Bogdan Rai\cb{t}\u{a}}}
\author[3]{{\Large  Matthew R.\,I.~Schrecker}}

\affil[1]{\small University of Oxford, Andrew Wiles Building,
 Woodstock Rd,
Oxford, OX2 6GG,
United Kingdom 
\protect\\
  {\tt{guerra@maths.ox.ac.uk}}\vspace{1em} \ }

\affil[2]{\small 
Max Planck Institute for Mathematics in the Sciences, Inselstra{\ss}e 22, 04103 Leipzig, Germany\protect\\
  {\tt{raita@mis.mpg.de}}
 \vspace{1em} \ }

\affil[3]{\small Department of Mathematics, University College London, 25 Gordon St, London, WC1H 0AY, UK
\protect\\
  {\tt{m.schrecker@ucl.ac.uk}} }

\date{}

\maketitle
\begin{abstract} 
For 
$l$-homogeneous linear differential operators $\cala$ of constant rank, we study the implication
$$
\begin{rcases}
v_j\rightharpoonup v&\textup{in }X\\
\cala v_j\rightarrow \cala v&\textup{in }W^{-l}Y
\end{rcases}
\implies 
F(v_j)\rightsquigarrow F(v)\textup{ in }Z,
$$
where $F$ is an $\cala$-quasiaffine function and $\rightsquigarrow$ denotes an appropriate type of weak convergence. Here $Z$ is a local $\lebe^1$-type space, either the space $\mathscr{M}$ of measures, or $\lebe^1$, or the Hardy space $\mathscr{H}^1$; $X,\, Y$ are $\lebe^p$-type spaces, by which we mean Lebesgue or Zygmund spaces. Our conditions for each choice of $X,\,Y,\,Z$ are sharp. Analogous statements are also given in the case when $F(v)$ is not a locally integrable function and it is instead defined as a distribution. In this case, we also prove $\mathscr{H}^p$-bounds for the sequence $(F(v_j))_j$, for appropriate $p<1$, and new convergence results in the dual of H\"older spaces when $(v_j)$ is $\cala$-free and lies in a suitable negative order Sobolev space $W^{-\beta,s}$. The choice of these H\"older spaces is sharp, as is shown by the construction of explicit counterexamples. Some of these results are new even for distributional Jacobians. 
\end{abstract}

{
\tableofcontents
}

\unmarkedfntext{
\hspace{-0.85cm} 


\noindent  \emph{Acknowledgments:} A.G. was supported by the Engineering and Physical Sciences Research Council [EP/L015811/1]. 
}

\section{Introduction}

The interaction between nonlinear functionals and weak convergence is a recurring theme in the study of nonlinear partial differential equations. Typically, one has a sequence of functions $(v_j)_j$ satisfying certain bounds leading to weak convergence and would like to understand the convergence properties of a given nonlinear function composed with the terms of the sequence. In particular, the theory of compensated compactness studies the weak convergence of $(F(v_j))_j$ for weakly convergent sequences $(v_j)_j$, where $F$ is a nonlinear function (for reasons we will see below, $F$ will always be taken to be a polynomial). In the absence of further constraints, it is easy to see that if $(v_j)_j$ converges weakly in some $\lebe^p$-space, the convergence
\begin{align}\label{eq:conv_intro}
F(v_j)\rightsquigarrow F(v)
\end{align}
cannot be expected to hold in any meaningful sense: one simply takes $F=|\,\cdot\,|^p$ and an oscillating or concentrating sequence $(v_j)_j$ to see that there need be no such relation in the limit between $\lim_{j\to\infty}F(v_j)$ and $F(v)$.

However, under certain conditions, such an identification of the limit
may hold true. The first known positive example is the weak continuity
of the Jacobian subdeterminants, under the restriction $v_j=\D u_j$
\cite{Ball1977, Morrey1952, Reshetnyak1967,Reshetnyak1968}. Subsequently, the
celebrated div-curl lemma \cite{Tartar1979} has led to remarkable developments which we will summarize roughly here as follows: under a compensation condition (of linear PDE type, relating to an operator $\cala$) on $(v_j)$, the convergence \eqref{eq:conv_intro} holds weakly-$*$ in the sense of distributions precisely for a class of (polynomial) nonlinearities $F$ determined by $\cala$, the so-called $\cala$-quasiaffine functions. 

This theory of compensated compactness rapidly found applications in
many areas of nonlinear PDE, covering hyperbolic conservation laws,
elasticity and geometric analysis, to name just a few areas. Tartar
\cite{Tartar1979} applied the div-curl lemma to demonstrate the
existence of entropy solutions to scalar hyperbolic conservation
laws. His techniques were extended by DiPerna \cite{DiPerna1983} to cover $2\times 2$ strictly hyperbolic systems and the isentropic Euler equations. Ball \cite{Ball1977}, M\"uller \cite{Muller1990} and others used the weak convergence of the determinant to prove existence theorems in nonlinear elasticity. Moreover, compensated compactness has been used to demonstrate regularity of harmonic maps between manifolds in \cite{Helein2002}, see also \cite{Evans1990} for an exposition.

Later, and in parallel to the development of the compensated
compactness theory, a remarkable observation by M\"uller led to a different
direction of study: it was shown in \cite{Muller1990} that if $v=\D
u\in\lebe^n_{\locc}(\R^n,\R^{n\times n})$ has non-negative
determinant, then $\det \D u\in L\log L_{\locc}$. This inspired
Coifman--Lions--Meyer--Semmes to prove in \cite{Coifman1993} that if
$\D u\in \lebe^n(\R^n,\R^{n\times n})$, then $\det \D u$ lies in the
Hardy space $\mathscr{H}^1$. A natural question to ask is what is the
class of pairs of compensating operators $\cala$ and nonlinearities
$F$. It was speculated and verified for several examples in
\cite{Coifman1993} that this is precisely the class of
$\cala$-quasiaffine functions. This claim was subsequently checked for
further examples in \cite{Lindberg2017} and, more recently, a positive result covering a large class of operators and exponents was given by the first two authors in \cite{Guerra2019}. To present this result accurately, we introduce some notation.

Given finite-dimensional inner product spaces $\mathbb{V}, \mathbb{W}$, we will consider a linear, $l$-th homogeneous constant coefficient operator
\begin{equation}
\label{eq:defA}
\mathcal A=\sum_{|\alpha|=l} A_\alpha \partial^\alpha, \hs \tp{where }A_\alpha \in \tp{Lin}(\mathbb V, \mathbb W),
\end{equation}
acting on maps $v\colon \Omega\subset\R^n\to \mathbb V$. In this
paper, $\Om\subset\R^n$ will always denote an open, bounded domain.
We define (using multi-index notation) $$\mathcal A(\xi)=\sum_{|\alpha|=l} A_\alpha\xi^\alpha,$$ a polynomial in $\xi \in \R^n$ and we will assume that $\mathcal A$ has \textbf{constant rank}, i.e.\ there is $r\in \N$ such that
\begin{equation}
\tp{rank}\, \mathcal A(\xi)=r \hs \tp{for all } \xi \in \mathbb{S}^{n-1}.
\label{eq:CR}
\end{equation}
The reader may find other characterizations of constant rank operators in \cite{GuerraRaita2020,Raita2018}.
We shall also make the non-degeneracy assumption that
\begin{equation}
\Lambda_{\mathcal A}\equiv \bigcup_{\xi \in \mathbb S^{n-1}} \ker \cala
(\xi)\quad \tp{spans } \mathbb V;\label{eq:spanning}
\end{equation}
the set $\Lambda_\cala$ is called the \textbf{wave cone} of $\cala.$
We recall that a locally bounded measurable function $F\colon \mathbb V \to \R$ is said to be $\cala$-quasiaffine if and only if
$$
F(v_0)=\fint_{\mathbb T^n} F(v_0+v(x))\dif x$$
for all $v_0\in \mathbb V$ and $v\in\hold^\infty(\mathbb T^n)$ such that
$$\cala v=0,\qquad\fint_{\mathbb T^n} v(x)\dif x=0,
$$
where $\mathbb T^n=\R^n/\mathbb Z^n$ is the $n$-dimensional torus. It
is known from the work of Murat and Tartar \cite{Tartar1979} that these functions are necessarily polynomials.

In \cite{Guerra2019}, the following result was proved:
\begin{thm}\label{thm:CC-standard}
Let $\Omega\subset \R^n$ be a bounded open set. Let  $\mathcal A$ be as in (\ref{eq:defA}) and assume (\ref{eq:CR})
and (\ref{eq:spanning}) hold. Then for a homogeneous $\mathcal A$-quasiaffine integrand $F\colon \mathbb V \to \R$ of degree $s\geq 2$, we have the implication
\begin{equation}\label{eq:M_intro}
\begin{rcases}
v_j\weakto v &\text{ in }\lebe^s(\Om,\mathbb V),\\
\mathcal{A}v_j\to \mathcal{A}v &\text{ in } \sobo^{-l,s}(\Om,\mathbb W)
\end{rcases}\hs \implies \hs F(v_j)\wstar F(v) \text{ in }\mathscr{M}(\Omega).
\end{equation}
Moreover, the following bound holds:
\begin{equation*}
\|F(v)\|_{\mathscr{H}^1(\R^n)}\leq C\|v\|_{\lebe^s(\R^n)}^s\quad \text{for }v\in\lebe^s(\R^n)\text{ such that }\cala v=0,
\end{equation*}
which implies that, with an exact constraint imposed, $\mathcal{A}v_j=0$, we have the convergence
\begin{equation}\label{eq:H^1_intro}
v_j\weakto v \text{ in }\lebe^s(\R^n,\mathbb V)
\hs \implies \hs F(v_j)\wstar F(v) \text{ in }\mathscr{H}^1(\R^n).
\end{equation}
\end{thm}
In \cite{Guerra2019}, \eqref{eq:M_intro} was derived from a more general lower semicontinuity result. Here we shall give a simple direct proof, see Section \ref{sec:Hp} below.

We remark, however, that the two topologies used in \eqref{eq:M_intro} and \eqref{eq:H^1_intro} are different and that it is possible that the convergences claimed may be sensitive to these topologies. The aim of this paper is to show that this is indeed the case: if we allow for perturbations $(\cala v_j)_j$ that are compact in the strong topology of a space $W^{-l}Y$, it can indeed happen that $(F(v_j))_j$ converges in $\mathscr{M}$, but not in $\mathscr H^1$. This is already the case if $Y=\lebe^s$, see Example~\ref{ex:d}.

We note that also the weak convergence of $(F(v_j))_j$ in $\lebe^1$ to $F(v)$ is to be expected in the case when $(F(v_j))_j$ is equi-integrable, c.f.\ \cite{Conti2011} in the case of the standard $\div$-$\curl$ lemma. By the Dunford--Pettis theorem, this is a statement referring to the identification of the weak $\lebe^1$ limit, as the existence of a weakly convergent subsequence in $L^1$ already follows from the equi-integrability. Such a result requires only the very weak compactness assumption $\cala v_j\to \cala v$ in $W^{-l,1}$ in the constraint (see Theorem \ref{thm:main} below for the precise statement), an advantage exploited in both \cite{Conti2013a} and \cite{Schrecker2019}.

In view of these considerations, the theme of the first part of this paper is as follows: we work with $\cala$-quasiaffine functions $F$ and determine precise conditions on the convergence of $(v_j)_j,\,(\cala v_j)_j$ to ensure that $(F(v_j))_j$ converges suitably in $\mathscr{M}$, $\lebe^1$, or $\mathscr{H}^1$. Moreover, since $\mathscr{H}^1_{\locc}$ contains all functions with $\lebe\log\lebe_{\locc}$ integrability, we will show that the differences in convergence can be seen very clearly on the scale of the Zygmund (or, more precisely, Orlicz) spaces $L^p\log^\alpha L$. For the sake of clarity of exposition, we give only local statements to avoid overburdening the reader with unnecessary details. The interested readers will have no difficulty in applying our results to related scenarios.

In the final sections of this paper, we investigate the case below integrability or below differentiability, that is, working with an $s$-homogeneous $\cala$-quasiaffine function $F$ and a sequence $(v_j)_j$ bounded in either $L^q$ with $q<s$ or in $W^{-\beta,s}$ with $\beta>0$. We investigate first conditions sufficient for a bound in the Hardy space $\mathscr{H}^{\frac{q}{s}}$, and secondly for the convergence of $F(v_j)$ to hold in the dual of a homogeneous H\"older space. 

\subsection{$L^1$-type spaces: $\mathscr M$ vs. $\lebe^1$ vs. $\mathscr{H}^1$}\label{sec:intro_L1H1}
We suppose $F$ is an $\cala$-quasiaffine integrand of $s$-growth. Our first main theorem is then as follows.
\begin{bigthm}\label{thm:main}
Let $\cala$ and $F$ satisfy the assumptions of Theorem~\ref{thm:CC-standard} where
 $s\geq 2$. 
 
 Take $\alpha\geq 0$ and let $r\geq 1$, $\beta\in \R$ be such that $L^r\log^\beta L_{\locc}\subseteq L^1_{\locc}$.
Consider a sequence $(v_j)_j$ such that
        \begin{align*}
           & v_j\rightharpoonup v \text{ in } L^s \log^\alpha L_{\locc}(\R^n,\mathbb{V})\\
           & \cala v_j\rightarrow \cala v\text{ in }\sobo^{-l}_{\locc}L^r \log^\beta L(\R^n,\mathbb W).
        \end{align*}
        The following hold:
        \begin{enumerate}
            \item \label{itm:M}
            If $L^r\log^\beta L_{\locc}\subseteq L^s_{\locc}$, then $F(v_j)\wstar F(v)$ in $\mathscr M_{\locc}(\R^n)$.
        \item\label{itm:L1}  
		If $\alpha> 0$, then $F(v_j)\rightharpoonup F(v)$ in $L^1_{\locc}(\R^n)$.
          \item\label{itm:H1} 
 If $\alpha\geq 1$ or  $L^r\log^\beta L_{\locc}\subseteq L^s \log^s L_{\locc}$, then $F(v_j)$ is bounded in $\mathscr H^1_{\locc}(\R^n)$.
          \end{enumerate}
          More generally, the conclusion of \ref{itm:L1} holds whenever $F(v_j)$ is equi-integrable.
\end{bigthm}
We refer the reader to Section \ref{sec:zygmund} for the definition of $W^{-l}L^p \log^\alpha L(\R^n)$; Section \ref{sec:FunctionSpaces} contains more details on the Zygmund spaces $L^p\log^\alpha L$ and Hardy spaces.

We give a precise estimate for the Hardy bound of Theorem \ref{thm:main}\ref{itm:H1} in Theorem \ref{thm:hardy_bound} (see also the Appendix).
The statements in Theorem \ref{thm:main}\ref{itm:M},\ref{itm:L1} are sharp, as can be seen from the examples in Section~\ref{sec:main}. There, we show that for the classical example $\cala(v,\tilde v)=(\tp{div\,} v,\tp{curl\,}\tilde v)$ and $F(v,\tilde v)=v\cdot \tilde  v$, in each of the first two cases above, the failure of the assumption can lead to the failure of the convergence or boundedness.

\begin{rmk}\label{rmk:sharpen}
For the sake of completeness, we remark that in fact Theorem~\ref{thm:main}\ref{itm:H1} may be improved in two ways:
\begin{enumerate}
    \myitem{(i)} In the case when $\alpha>1$, the statement in Theorem~\ref{thm:main}\ref{itm:H1} can be replaced with the slightly finer statement that
$
F(v_j)\rightharpoonup F(v)\text{ in }\lebe\log\lebe_{\locc}.
$
\myitem{(ii)}\label{itm:sharp} The two conditions given in Theorem \ref{thm:main}\ref{itm:H1} may actually be interpolated. Such an interpolation is best stated in the language of Orlicz functions and so we leave this to an Appendix. The sharp result is stated precisely in Theorem~\ref{thm:hardy_boundOrlicz}.
\end{enumerate}
\end{rmk}

For better readability of the paper, we present the first main result on the Lebesgue scale:

\begin{bigcor}\label{cor:Lp}
         Let $1\leq q,\,r<\infty$ be such that $q\geq s$. Consider a sequence $(v_j)_j$ such that
        \begin{align*}
           & v_j\rightharpoonup v \text{ in }\lebe^{q}_{\locc}(\R^n,\mathbb{V})\\
           & \cala v_j\rightarrow \cala v\text{ in }\sobo^{-l,r}_{\locc}(\R^n,\mathbb W).
        \end{align*}
        Then:
        \begin{enumerate}
        \item
        If $r\geq s$, then $F(v_j)\wstar F(v)\text{ in } \mathscr{M}_{\locc}(\R^n).$
        \item
         If $q>s$, then $F(v_j)\rightharpoonup F(v)\text{ in } \lebe^1_{\locc}(\R^n).$
         \item \label{itm:H1Lebe}
          If $r>s$ or $q>s$, then $\left(F(v_j)\right)_j\text{ is bounded in } \mathscr H^1_{\locc}(\R^n).$
          \end{enumerate}
\end{bigcor}
Minor adaptations of the examples in Section~\ref{sec:main} show that these statements are sharp as well (on the Lebesgue scale). Note that in both Theorem~\ref{thm:main}\ref{itm:H1} and Corollary~\ref{cor:Lp}\ref{itm:H1Lebe}, the $\mathscr{H}^1_{\locc}$ boundedness can be used to strengthen the convergence of $F(v_j)$ to $F(v)$ in the following sense: if $\rho\in\hold^\infty_c(\R^n)$, then
$$
\rho F(v_j)\wstar \rho F(v)\text{ in } h^1(\R^n),
$$
where we refer to Section \ref{sec:FunctionSpaces} for the definition of the Goldberg--Hardy space $h^1(\R^n)$; in particular, one can test the convergence against functions in a local version of $\tp{VMO}(\R^n)$, known as $\tp{vmo}(\R^n)$. Local variants of the div-curl lemma and the relationship between $h^1$ and $\tp{vmo}$ were discussed previously in \cite{Dafni2002,Dafni2005}. A variant of the Hardy space bound on domains, taking account of boundary values, was proved in \cite{Hogan2000}.

\subsection{Distributional null Lagrangians and $\mathscr{H}^p$ spaces}
As is well known from \cite{Ball1977}, even if the quantity $F(v)$ is not integrable, one can define suitable distributional variants. In particular, in \cite[Remark~7.4]{Guerra2019} the quantity $F(v)$ was, roughly speaking, defined for $\cala$-free fields $v$, as the distributional limit of the quantities $F(v_j)$, where $v_j$ is a smooth, $\cala$-free approximation of $v$. Here, we will allow for a perturbation in the constraint (i.e.~we will require convergence of $\cala v_j$, but not that $\cala v_j=0$ for all $j$), but work with a slightly less general set up than in Section~\ref{sec:intro_L1H1}, for better readability of the paper. Alternative statements can be obtained with minor alterations of our methods. 

We freeze a bounded open set $\Omega\subset\R^n$, an $\cala$-quasiaffine integrand $F$ of $s$-growth, and the exponents
\begin{align*}
    q\in\left(\frac{ns}{n+1},s\right),\quad r_*\equiv \frac{q}{q-s+1},
\end{align*}
so that $\frac{s-1}{q}+\frac{1}{r_*}=1$. We will essentially define $F(v)$ for fields $v\in\lebe^q$ with support in $\Omega$:
\begin{defn}\label{def:distribnulllag}
Let $(v_j)_j\subset C^\infty_c(\Omega,\mathbb V)$ be such that 
\begin{align*}
           & v_j\rightharpoonup v \text{ in }\lebe^{q}(\Om,\mathbb{V})\\
           & \cala v_j\rightarrow \cala v\text{ in }\sobo^{-l,r_*}(\Om,\mathbb W).
        \end{align*}
        Then we define
        $$
        F(v)\equiv \textup{w*-}\lim_{j\rightarrow\infty} F(v_j)\textup{ in }\mathscr{D}^\prime(\R^n).
        $$
\end{defn}
The following theorem shows that $F(v)$ above is well defined. Moreover, we obtain a Hardy space bound generalizing the results of \cite{Coifman1993} (see also \cite{Bonami2010} for a variant):
\begin{bigthm}\label{thm:dist}
Let $\Omega\subset\R^n$ be a bounded open set, $F\colon \mathbb V\rightarrow \R$ be an homogeneous $\cala$-quasiaffine map of degree $s\geq2$, and $ns/(n+1)<q\leq s$, $r\geq r_*$. For a sequence $(v_j)_j\subset C^\infty_c(\Omega,\mathbb V)$,
        \begin{align*}
\begin{rcases}
 v_j\rightharpoonup v & \text{ in }\lebe^{q}(\Omega,\mathbb{V})\\
\cala v_j\rightarrow \cala v & \text{ in }\sobo^{-l,r}(\Omega,\mathbb W)
\end{rcases}
\quad \implies \quad
F(v_j)\wstar F(v)\text{ in }\mathscr{D}^\prime (\Omega).
        \end{align*}
        Moreover, if $\cala v$ is bounded in $W^{-l}L^{r_*}\log^{r_*}L$ (or simply $r>r_*$), we also have that $(F(v_j))_j$ is bounded in $\mathscr{H}^{q/s}(\R^n)$ with the estimate
        $$
        \|F(v_j)\|_{\mathscr H^{\frac{q}{s}}}\leq C\left(\|v_j\|_{\lebe^q}+\|\cala v_j\|_{W^{-l}L^{r_*}\log^{r_*}L}\right)^s.
        $$
\end{bigthm}
Aside from the convergence and boundedness results concerning distributional quantities, while proving Theorem~\ref{thm:dist} in Section~\ref{sec:Hp}, we also include a new proof of \eqref{eq:M_intro}. A third proof will be sketched in Remark~\ref{rmk:CC}.

 It was observed in \cite{Brezis2011b} that, as well as being defined below integrability, in fact the distributional Jacobian could be defined below differentiability in the following sense. Defining the Jacobian operator $J(u)=\det(\D u)$ formally (for smooth functions, say), it is clear that $J(u)$ is meaningful for $u\in W^{1,n}$. \cite[Theorem 3]{Brezis2011b} proved that $J(u)$ could  be given a distributional meaning even for $u\in W^{\frac{n-1}{n},n}$, that is, with differentiability strictly below the usual setting. In fact, they prove a stronger result: that the distributional Jacobian defined on this space takes values in the dual of the homogeneous Lipschitz space $\hold^{0,1}$, and an estimate on the difference of the Jacobians of functions $u$ and $v$ in the dual Lipschitz norm may be obtained in terms of the $W^{\frac{n-1}{n},n}$ difference of $u$ and $v$.

In our final main theorem, we prove firstly that this result extends to more general $s$-homogeneous $\cala$-quasiaffine functions, and secondly that under differentiability conditions between the critical $\frac{s-1}{s}$ and $1$, the estimate may be strengthened to one in the dual of a homogeneous H\"older space.  We remark that some estimates for the Jacobian in fractional Sobolev spaces $W^{-\beta,p}$ have been proved in \cite[Theorem 1.1]{Lenzmann2018} by adapting the methods of \cite{Brezis2011b}.

\begin{bigthm}\label{thm:D}
Let $\Omega\subset\R^n$ be either a bounded Lipschitz domain or $\Omega=\R^n$, $F\colon \mathbb V\rightarrow \R$  an homogeneous $\cala$-quasiaffine map of degree $s\geq2$. Suppose that $u,v\in C^\infty_c(\Omega,\mathbb V)$ are $\cala$-free. Let $\alpha\in (0,1]$ and $\beta=1-\frac \alpha s$. Then, for any $\varphi\in C^{0,\alpha}(\overline{\Omega})$, we obtain the estimates
\begin{align*}
\Big|\int_{\Omega} \big(F(u)-F(v)\big)\varphi\,\dif x\Big|\leq C[\varphi]_{C^{0,\alpha}}[u-v]_{W^{-1+\beta,s}}\big([u]_{W^{-1+\beta,s}}+[v]_{W^{-1+\beta,s}}\big)^{s-1}.
\end{align*}  
\end{bigthm}
Theorem \ref{thm:D} appears to be new even in the case of the Jacobian operator, see Theorem~\ref{thm:interpolated}. Earlier results due to \cite[Theorem 1]{Brezis2011b} for the Jacobian and \cite[Proposition 7.1]{Guerra2019} for $\cala$-quasiaffine functions controlled the difference $F(u)-F(v)$ in the dual Lipschitz norm in terms of the $L^p$ norms of $u$ and $v$ and their difference in some $W^{-1,q}$ (where $q$ and $p$ satisfy a suitable relation).  Such results may be partially recovered as corollaries of Theorem \ref{thm:D}, as explained in \cite{Brezis2011b} for the special case of the Jacobian, see Section \ref{sec:holder}. 

Finally, Theorem \ref{thm:D} is sharp on the scale of fractional Sobolev spaces. That is, we show for the Jacobian that on a bounded domain $\Omega$, if $W^{\be,p}(\Omega)\not\hookrightarrow W^{\frac{n-\al}{n},n}(\Omega)$, there are uniformly bounded sequences $(u_k)_k$ in $W^{\be,p}(\Omega)$ and $(\varphi_k)_k$ in $C^{0,\alpha}(\overline{\Omega})$ such that
$$\int_{\Omega}\det(\D u_k)\varphi_k\dif x\to\infty\quad \text{ as }k\to\infty.$$
In order to prove this, we split the analysis into three cases. The construction of the sequence $u_k$ in the two more complicated cases is modelled on that of \cite[Lemma 5]{Brezis2011b} and it is based on a frequency decomposition, adding appropriately weighted oscillations at increasing frequencies. To allow for the range of values of $\alpha$ that we must consider, the construction needs to be modified and, in particular, we need to consider a sequence of test functions, instead of a fixed test function as in \cite{Brezis2011b}. This sequence $\varphi_k$ is constructed through a further frequency decomposition. We refer the reader to Proposition \ref{prop:sharpjacobian} for more details and to \cite{Baer2015} for optimality results for the Hessian determinant.

Let us briefly describe the organisation of this paper. In \S\ref{sec:FunctionSpaces}, we provide basic definitions and results concerning Zygmund spaces as well as the Hardy spaces and their local variants. We state a useful version of the standard H\"ormander--Mihlin multiplier theorem adapted to the Zygmund spaces and finally state and prove an extension of the standard Lipschitz truncation to higher order operators. In \S\ref{sec:L^1}, we employ this Lipschitz truncation to prove a compensated compactness weak convergence result under a $W^{-l,1}$ compactness assumption. This is followed in \S\ref{sec:H1} by the proofs of the Hardy space estimates of Theorem \ref{thm:main}. In \S\ref{sec:main}, we use the results of \S\ref{sec:L^1}--\ref{sec:H1} to prove  Theorem \ref{thm:main} and give a series of counterexamples to demonstrate the sharpness of Theorem \ref{thm:main}. After the proof of Theorem~\ref{thm:dist} in \S\ref{sec:Hp}, we conclude the proof of Theorem \ref{thm:D} and related results, including the construction of the counterexamples mentioned above, in \S\ref{sec:holder}. A final Appendix states and demonstrates the sharp conditions for the Hardy bound of Theorem \ref{thm:main} in terms of Orlicz functions, as claimed in Remark \ref{rmk:sharpen}\ref{itm:sharp}.

\section{Function spaces and harmonic analysis}\label{sec:FunctionSpaces}

In this section we gather in a concise way some facts and definitions about functions spaces that we shall use.
For simplicity we will only work with Zygmund spaces, but more general 
versions of the results stated here 
can be found in the monograph of {Rao}--{Ren} \cite{Rao1991}. 
For results about harmonic analysis we refer the reader to the monograph \cite{Stein2016}.

\subsection{Zygmund spaces}
\label{sec:zygmund}

We will work extensively with the \textbf{Zygmund spaces} $\lebe^p \log^\alpha L(\Omega,\mathbb V)$, defined as the space of those measurable functions $f\colon \Omega\to \mathbb V$ such that
$
\int_\Omega |f(x)|^p \log^\alpha(1+|f(x)|) \dif x<\infty.
$
For $1< p<\infty$, $\al\in\R$ or $p=1$, $\al\geq 0$, $L^p\log^\alpha L$ is a Banach space under the norm
$$\Vert f \Vert_{\lebe^p \log^\alpha L(\Omega)}\equiv 
\left(\int_\Omega |f(x)|^p \log^\alpha \left(1+\frac{|f(x)|}{\Vert f \Vert_p}\right)\dif x\right)^\frac 1 p, 
$$
see e.g.\ the appendix in \cite{Iwaniec1999b}.
It is convenient to record the following fact concerning duals of Zygmund spaces, 
see \cite[Theorem 8.4]{Bennett1980}:

\begin{thm}\label{thm:zygmunddual}
Let $1<p<\infty$ and $\alpha \in \R$. Then the dual of the Banach space $\lebe^p\log^\alpha L$ can be identified, up to an equivalence of norms, with $\lebe^{p'}\log \lebe^{-\frac{p'}{p}\alpha}$, where $\frac 1 p + \frac{1}{p'}=1$.
\end{thm}

In particular, we have the Zygmund version of H\"older's inequality,
\begin{equation}
\label{eq:holder}
\Vert f_1\dots f_m \Vert_{L^p\log^\alpha L}\lesssim \Vert f_1\Vert_{L^{p_1}\log^{\alpha_1} L}\dots
\Vert f_m\Vert_{L^{p_m}\log^{\alpha_m} L},
\end{equation}
which holds whenever $p_i>1, \alpha_i\in \R$ are such that 
$\frac 1 p =\frac{1}{p_1} + \dots +\frac{1}{p_m}$ and $\frac \alpha p= \frac {\alpha_1 }{p_1}+ \dots + \frac {\alpha_m }{p_m} $.

For later use, it will be important to deal with \textbf{Zygmund--Sobolev} spaces, defined as an extension of Sobolev spaces. 

\begin{defn}\label{def:zygmundsobolev}
Let $p>1$ and $k\in \N$. Then:
\begin{enumerate} 
\item $W^{k}L^p \log^\alpha L(\Omega,\mathbb V)$ is the space of those distributions $f \in \mathscr{D}'(\Omega,\mathbb V)$ such that, for all multi-indices $\alpha$ with $|\alpha|\leq k$, we have $\partial^\alpha f\in L^p\log^\alpha L(\Omega,\mathbb V)$;
\item\label{it:negsobspace} $W^{-k}L^p\log^\alpha L(\Omega,\mathbb V)$ is the space of $k$-th order distributional derivatives of functions in $L^p \log^\alpha L(\Omega, \mathbb V)$;
\item ${\dot W}^{-k}L^p\log^\alpha L(\Omega, \mathbb V)$ is the homogeneous version of \ref{it:negsobspace} and is defined as the space of $f\in \mathscr S'(\R^n)$ which have support in $\Omega$ and
\begin{equation*}
\label{eq:homsob}
\Vert f \Vert_{{\dot W}^{-k}L^p \log^\alpha L(\R^n)}\equiv \left\Vert \mathcal F^{-1}\left(\frac{\mathcal F f(\xi)}{|\xi|^k}\right)\right\Vert_{L^p\log^\alpha L(\R^n)}<\infty.
\end{equation*}
\end{enumerate}
The local versions of these spaces are defined in the obvious way.
\end{defn}

To conclude this subsection we recall that it is possible to interpolate results from the scale of Lebesgue spaces to that of Zygmund spaces. In particular, 
$L^p$-multipliers 
extend to bounded operators on Zygmund spaces, see \cite[\S 12.12]{Iwaniec2001}. The analogue of the classical H\"ormander--Mihlin multiplier theorem, in this generalized setting, is the following:

\begin{thm}\label{thm:HM}
    Let $m$ be a H\"ormander--Mihlin multiplier, so $m$ corresponds to a Calder\'on--Zygmund operator $T_m$ which, for any $p\in(1,\infty)$,
    is a bounded operator $T_m\colon L^p(\R^n)\to L^p(\R^n)$.
        Then, for $p\in (1,\infty)$ and $\alpha\in \R$, there is a constant $C=C(p,\alpha)>0$ such that
        $$\Vert T_m f \Vert_{L^p \log^\alpha L(\R^n)} \leq C\Vert f \Vert_{L^p \log^\alpha L(\R^n)}.$$
\end{thm}

\subsection{Hardy spaces, their local versions, and duality}

We fix a test function $\phi$ such that $\int_{\R^n} \phi\neq 0$. Given a distribution $f \in \mathscr D'(\R^n)$, we define the \textbf{Hardy--Littlewood maximal function}, together with its local version, by
\begin{align*}
\mathcal M f(x)& \equiv \sup_{0<t<\infty} |f*\phi_t|(x),\\
\mathcal M_\tp{loc}f(x)& \equiv \sup_{0<t<1} |f*\phi_t|(x).
\end{align*}
For a number $0<p\leq \infty$, the real \textbf{Hardy space} $\mathscr H^p(\R^n)$ is
\begin{equation}
\mathscr{H}^p(\R^n)\equiv \{f \in \mathscr S'(\R^n): \mathcal M f \in \lebe^p(\R^n)\}
\label{eq:maximaldef}
\end{equation}
and it is equipped with the (quasi-)norm $\Vert f \Vert_{L^p(\R^n)}\equiv \Vert \mathcal M f\Vert_{L^p(\R^n)}$.
The maximal theorem shows that, for $p>1$, $\mathscr H^p(\R^n) \cong L^p(\R^n)$. The case $p=1$ is particularly important and we have $\mathscr H^1(\R^n)\subsetneq L^1(\R^n)$.
The \textbf{local Hardy space} is defined by
$$\mathscr{H}^1_\tp{loc}(\R^n)\equiv \{f\in \lebe^1_\tp{loc}(\R^n): \mathcal M_\tp{loc}f \in \lebe^1_\tp{loc}(\R^n)\}.$$
Similarly, the \textbf{Goldberg--Hardy space} \cite{Goldberg1979} is defined by
$$h^1(\R^n)\equiv\{f \in \lebe^1(\R^n):\mathcal  M_\tp{loc}f \in \lebe^1(\R^n)\}.$$

A celebrated theorem due to Fefferman--Stein \cite{Fefferman1972} shows that the definition of $\mathscr{H}^p(\R^n)$ does not depend on the choice of $\phi$, and neither do the definitions of its local variants.
Furthermore, it is easy to see, c.f.\ \cite[Lemma 5.1]{Evans1994a}, that given $f \in \mathscr{H}^1_\tp{loc}(\R^n)$ and $\phi \in C^\infty_c(\R^n)$, $\phi f\in h^1(\R^n)$.

Our interest in local Hardy spaces comes from the following classical result \cite{Stein1969}:

\begin{prop}\label{prop:stein}
Let $f\in L^1_{\locc}(\R^n)$. If $f \in L\log L_\tp{loc}(\R^n)$ then $f \in \mathscr{H}^1_\tp{loc}(\R^n).$ The converse also holds if $f\geq 0$.
\end{prop}

The reader may find further information on local Hardy spaces in  \cite{Dafni2002, Dafni2005, Semmes1994}. 

We also recall that both $\mathscr H^1(\R^n)$ and $h^1(\R^n)$ are endowed with a genuine weak-$*$ topology, induced from their preduals (more generally, the balls in $\mathscr H^p(\R^n)$ are weakly-$*$ precompact in the sense of distributions). We refer the reader to \cite{Dafni2002} for the definition of the space of functions of \textbf{vanishing mean oscillation}
$\tp{VMO}(\R^n)$
and its local version $\tp{vmo}(\R^n)$. Moreover, see \cite{Dafni2002, Sarason1975},
$$\mathscr H^1(\R^n) = \tp{VMO}(\R^n)^*,
\qquad 
h^1(\R^n)= \tp{vmo}(\R^n)^*.$$

\subsection{Further properties of constant rank operators}
An important tool used in \cite{Guerra2019} was the observation from \cite{Raita2018} that a linear differential operator $\cala$ as defined in \eqref{eq:defA} satisfies Murat's constant rank condition \eqref{eq:CR} if and only if there exists, for some $k\in\N$, a $k$-homogeneous linear partial differential operator $\mathcal B$ with constant coefficients,
$$\mathcal B\equiv \sum_{|\beta|=k} B_\beta \partial^\beta,\qquad 
B_\beta \in \text{Lin}(\mathbb U,\mathbb V),$$
where $\mathbb U$ is a finite-dimensional inner product space and such that
$$
\ker \cala(\xi)=\mathrm{im\,}\B(\xi)\quad \text{for }\xi\in\R^n\setminus\{0\}.
$$
This does not imply a Poincar\'e lemma, i.e.\ it is not the case that $\cala v=0 \implies v=\B u$ over simply-connected domains. However, such an implication holds whenever we have access to Fourier analysis. This is  reflected in the Helmholtz--Hodge type decomposition of Proposition~\ref{prop:HHHH}, which can be summarized as follows: for test functions $v\in C^\infty_c(\R^n,\mathbb V)$, we have a decomposition
$$
v=\B u+\cala^*w,\quad \|\D^ku\|_X\lesssim\|\B u\|_X\lesssim\|v\|_X,\quad \|\cala^*w\|_Y\lesssim \|\cala v\|_{W^{-l}Y},
$$
where $X,\,Y$ are spaces on which zero-homogeneous multiplier operators are bounded.

The other important observation that we will use is that we can write in jet notation
$$
\B u=T(\D^ku),
$$
where $T$ is a tensor. Consequently, certain statements pertaining to the $\cala$-free setting can be reduced to the setting of higher order gradients. For instance, under the spanning cone condition \eqref{eq:spanning}, we have from \cite[Lemma~5.6]{Guerra2019} that $F$ is $\cala$-quasiaffine if and only if $F\circ T$ is $k$-quasiaffine. The latter class is known from \cite{Ball1981}.

\subsection{Lipschitz truncation}

In this section, we let $p\in [1,\infty)$. The following proposition is an extension to higher order operators of the standard maximal function argument  \cite{Evans2015,Zhang1992}. As well as being of independent interest, Lipschitz truncation will be useful in \S\ref{sec:L^1} in order to prove the compensated compactness convergence result for constraints in $L^1$-type spaces.

\begin{prop}\label{prop:Lipschitztruncation}
Let $v\in W^{k,p}(\R^n,\mathbb V)$ and $\lambda>0$ be arbitrary. There is $u \in W^{k,\infty}(\R^n,\mathbb V)$ such that
\begin{align}
\Vert \tp D^k u\Vert_{L^\infty} &\leq C \lambda\\
\label{eq:volumeestimate}
|\{ v \neq u \}| &\leq C \frac{1}{\lambda^p} \int_{\{|v|+\dots + |\tp D^k v|>\lambda\}} |v|^p+ \dots + |\tp D^k v|^p \dif  x
\end{align}
where the constants $C>0$ only depend on $k,n,p, \mathbb V$.
\end{prop}

In fact, we also have the estimate
$$
\Vert v-u \Vert_{W^{k,p}}^p \leq C \int_{\{|v|+\dots + |\tp D^k v|>\lambda\}} |v|^p+ \dots + |\tp D^k v|^p \dif x.
$$
Recall that for weakly differentiable functions $u,v$ and a measurable set $E$, we have
$$u = v \tp{ a.e.\ in } E \hs \implies \hs \tp Du=\tp Dv \tp{ a.e.\ in } E.$$
In particular, for Sobolev functions $u,v\in W^{k,p}$,
\begin{equation}
\bigcup_{i=1}^k\{\tp D^i v\neq\tp  D^i u\}\subset \{v \neq u\}.
\label{eq:levelset}
\end{equation}

\begin{proof}
By working componentwise, after a choice of basis of $\mathbb V$, we can assume that $\mathbb V=\R$. Note also that, by approximation, we can assume that $v$ is smooth and compactly supported on a ball $B_R$.
Let us write $f\equiv |u|+ |\tp Du|+ \dots+|\tp D^k u|$ and consider the set 
$$E_\lambda \equiv \left\{ x\in \R^n: \textup{ there exists } r>0 \textup{ such that }\fint_{B_r(x)} f(y) \dif y\geq 2\lambda  \right\}.$$
An application of the Vitali Covering Theorem shows that the volume of $|E_\lambda|$ satisfies (\ref{eq:volumeestimate}): indeed, one can find a countable collection of disjoint balls $B_{r_i}(x_i)$ such that $B_{5r_i}(x_i)$ covers $E_\lambda$ and $\fint_{B_{r_i}(x_i)} f \dif y \geq 2 \lambda$. This implies the estimate
$$r_i^n \leq C \frac{1}{\lambda} \int_{B_{r_i}(x_i) \cap \{f>\lambda\}} f \dif y$$
from which one deduces, with the help of H\"older's and Markov's inequalities, that
$$|E_\lambda|\leq C 5^n \sum_{i=1}^\infty r_i^n \leq C \frac 1 \lambda \int_{\{f>\lambda\}} f \dif y \leq C \frac{1}{\lambda^p} \int_{\{f> \lambda\}} |v|^p + \dots + |\tp D^k v|^p \dif y,$$
as desired.

Let us further enlarge $E_\lambda$ by a null set, so that all points in $\R^n\backslash E_\lambda$ are Lebesgue points of $f$. Consider  the function $$h_r(x)\equiv \fint_{B_r(x)} |\tp D^k v(x)-\tp D^k v(y)| \dif y$$
which, for $x\not \in E_\lambda$, goes to zero pointwise as $r\to 0$ by Lebesgue's differentiation theorem. By Egorov's Theorem, there is a set $F_\lambda$ with measure $|F_\lambda|\leq |E_\lambda|$ and such that, in $B_R\backslash(E_\lambda \cup F_\lambda)$, the function $h_r$ goes to zero uniformly as $r\to 0$. Let us thus write $K_\lambda\equiv E_\lambda \cup F_\lambda$; then there is a function $\omega\colon (0,\infty) \to (0,4 \lambda]$ such that
\begin{equation}\label{eq:boundh}|h_r(x)|\leq \omega(r) \hspace{0.5cm} \textup{for all } x \not\in K_\lambda\end{equation}
and moreover $\omega(r)\to 0$ as $r\to 0$.

We write $P^k_x[v]$ for the $k$-th order Taylor polynomial of $v$ centered at $x$. 
Note that, for $x\not \in K_\lambda$, we have the estimate
$$\fint_{B_r(x)} |v(y)-P^k_x[v](y)|\dif y\leq C r^k \omega(r).$$
This follows simply by integrating Taylor's formula
$$v(y)-P^k_x[v](y)= \sum_{|\alpha|=k} \frac{k}{\alpha!} (y-x)^\alpha \int_0^1(1-t)^{k-1}\partial^\alpha \left[v(t x + (1-t)y)-v(x)\right]\dif t$$
in $y$ and using (\ref{eq:boundh}). Furthermore, it follows from the triangle inequality and the above estimate that, whenever  $x,x'\not \in K_\lambda$ are such that $|x-x'|=r$, we further have
$$\fint_{B_r(x) \cap B_{r'}(x')} \left|P_x^k[v]-P^k_{x'}[v]\right|(y)\dif  y\leq C r^k \omega(r).$$
This estimate in fact yields a uniform estimate on the difference of the two polynomials $P^k_x[v]$ and $P^k_{x'}[v]$ over $B_R\backslash K_\lambda$. We therefore satisfy the necessary conditions to apply Whitney's extension theorem to see we can find a function $u$ with the required properties.
\end{proof}

\section{Weak continuity under a $W^{-l,1}$-compactness assumption}\label{sec:L^1}

The main result of this section is Theorem \ref{thm:CDM}, which is an extension of the result in \cite{Conti2011}. 
We begin with the following lemma:

\begin{lemma}\label{lema:aux}
Let $p\in(1,\infty)$ and let $v_j\in \lebe^p(\Om,\mathbb V)$ be a $p$-equi-integrable sequence, i.e.~suppose that there exists an increasing function $\om\colon[0,\infty)\to[0,\infty)$ such that $\lim_{m\to0}\om(m)=0$ and
$$\int_{E}|v_j|^p\,dx\leq\om(m) \text{ for all measurable sets $E$ such that }|E|\leq m.$$
If $\mathcal{A}v_j\to\mathcal{A}v$ in $W^{-l,1}(\Omega,\mathbb W)$, then $\mathcal{A}v_j\to\mathcal{A}v$ also in $W^{-l,p}(\Omega, \mathbb W)$.
\end{lemma}

\begin{proof}
Without loss of generality, we suppose that $v=0$.
Note that
$$\|\mathcal{A}v_j\|_{W^{-l,p}(\Om)}=\sup\left\{\left|\int_\Om\mathcal{A}^*u\cdot v_j\dif x \right|:\,u\in C_0^\infty(\Om),\,\|u\|_{W^{l,p'}(\Om)}\leq 1\right\}.$$
Let $u$ be such a function and extend it by zero to $\R^n$. Applying Proposition \ref{prop:Lipschitztruncation} with $\la>0$, to be chosen later, we obtain a function $w\in W^{l,\infty}(\R^n)$ such that
\begin{align*}
&\|\tp \D^lw\|_{\lebe^\infty}\leq C\la,\\
&|\{w\neq u\}|\leq C \la^{-p'}\int_{| u |+\cdots+|\tp \D^l u |>\la}| u |
+\dots+|\tp \D^l u |
\dif x.
\end{align*}
Let $E\equiv \{ w \neq u \}$ and recall from (\ref{eq:levelset}) that $\{\D^l w \neq \D^l u \}\subseteq E$.
Then
\beqas
\int_\Om\mathcal{A}^* u \cdot v_j\,\dif x= \int_E(\mathcal{A}^* u -\mathcal{A}^* w )\cdot v_j\,\dif x+\int_\Om\mathcal{A}^* w  \cdot v_j\,\dif x
=\tp{ I}+\tp{II}.
\eeqas
Estimating $\tp{II}$ first,
$$|\tp{II}|\leq C \|\D^l w \|_{\lebe^\infty}\|\mathcal{A}v_j\|_{W^{-l,1}(\Om)}\leq C\la\|\mathcal{A}v_j\|_{W^{-l,1}(\Om)}.$$
For $\tp{I}$ we have
\beqas
|\tp{I}|& \leq C \Big(\int_E(|\mathcal{A}^* u |+\la)^{p'}\,\dif x\Big)^{\frac{1}{p'}}\Big(\int_{E}|v_j|^p \,\dif x\Big)^{\frac{1}{p}}\\
&\leq C \big(1+\la|E|^{\frac{1}{p'}}\big)\om(|E|)^{\frac{1}{p}}\\
&\leq C \om(C\la^{-p'})^{\frac{1}{p}}.
\eeqas
Thus choosing
$$\la=\|\mathcal{A}v_j\|_{W^{-l,1}}^{-\half},$$
we have that, as $j\to\infty$, $|\tp{I}|+|\tp{II}|\to0$ and the lemma is proved.
\end{proof}

\begin{thm}\label{thm:CDM}
Suppose that $s,\, \mathcal{A}$ and $F$ are as in Theorem \ref{thm:CC-standard}.  Then 
$$
\begin{rcases}v_j\weakto v &\text{  in }\lebe^s(\Om, \mathbb V)\\
\mathcal{A}v_j\to \mathcal{A}v &\text{  in } \sobo^{-l,1}(\Om, \mathbb W)\\
F(v_j)\weakto \ell &\text{  in } \lebe^1(\Om)
\end{rcases}\hs \implies \hs 
F(v_j)\weakto F(v)=\ell  \text{ in }\lebe^1(\Om).
$$
\end{thm}

\begin{proof}
As $|v_j-v|^p$ is a bounded sequence in $\lebe^1$, we may apply the biting lemma \cite{Ball1989a}, to obtain a sequence of measurable sets $E_j\subset\Om$ such that $|E_j|\to0$ and, up to a subsequence, $|v_j-v|^p\mathbbm{1}_{E_j^c}$ is equi-integrable. Let
$$\tilde v_j=v_j\mathbbm{1}_{E_j^c}+v\mathbbm{1}_{E_j}.$$
Now $\tilde v_j-v=(v_j-v)\mathbbm{1}_{E_j^c}$ is $p$-equi-integrable.
Also, we have the convergence
$$\|\tilde v_j-v_j\|_{\lebe^1(\Om)}=\|v-v_j\|_{\lebe^1(E_j)}\leq |E_j|^{\frac{1}{p'}}\|v_j-v\|_{\lebe^p(\Om)}\to0,$$
i.e.\ $\tilde v_j-v_j\to 0$ strongly in $\lebe^1(\Om)$. Thus, in particular, we also have that $\tilde v_j\weakto v$ in $\lebe^p(\Om)$.
Moreover, $\D^l(\tilde v_j-v_j)\to0$ strongly in $W^{-l,1}$, and thus $\mathcal{A}\tilde v_j\to\mathcal{A}v$ strongly in $W^{-l,1}$ as well. 

Next, we note that 
$$F(\tilde v_j)-F(v_j)=\big(F(v)-F(v_j)\big)\mathbbm{1}_{E_j}.$$ 
It is clear that $F(v)\mathbbm{1}_{E_j}\to0$ strongly in $\lebe^1$, and so we use the equi-integrability of $F(v_j)$ to observe that
 $$\int_\Om|F(v_j)|\mathbbm{1}_{E_j}\,\dif x=\int_{E_j}|F(v_j)|\,\dif x\to0.$$
Putting these together, we see that
 $$F(\tilde v_j)-F(v_j)\to0 \text{ strongly in $\lebe^1$.}$$
By Lemma \ref{lema:aux}, as we have that $\tilde v_j-v=(v_j-v)\mathbbm{1}_{E_j^c}$ is $p$-equi-integrable and that $\mathcal{A}(\tilde v_j-v)\to0$ in $W^{-l,1}$, we also have that
$$\mathcal{A}(\tilde v_j-v)\to0 \text{ strongly in }W^{-l,p}(\Omega,\mathbb W).$$
Therefore, by Theorem \ref{thm:CC-standard}, we have that, up to a subsequence,
$$F(\tilde v_j)\wstar F(v) \text{ in }\mathcal{D}'(\Omega).$$
As we know that $F(v_j)\weakto\ell$ in $\lebe^1$ and also $F(\tilde v_j)-F(v_j)\to0$ strongly in $\lebe^1$, we obtain that $\ell=F(v)$ by uniqueness of weak limits. Moreover, as the limit is independent of subsequence taken, we have that the limit holds along the entire sequence.
\end{proof}

\section{$\mathscr H^1$ estimates and local $L \log L$-integrability}\label{sec:H1}

We begin this section with a modification of the Helmholtz--Hodge decomposition of \cite{Guerra2019}.
\begin{prop}\label{prop:HHHH}
Let $1<s<\infty$ and let $v\in \lebe^s(\R^n;\mathbb{V})$ satisfy
$\mathcal{A}v\in {W}^{-l}L^s \log^s L(\R^n; \mathbb W).$
There are $u\in W^{k,s}(\R^n;\mathbb{U})$, $w\in W^{l}L^s\log^s L(\R^n;\mathbb{W})$ such that
$$v =\mathcal{B}u+\mathcal{A}^*w.$$
Moreover,
$$\|\mathcal{B}u\|_{\lebe^s(\R^n)}\leq C\|v\|_{\lebe^s(\R^n)},\quad \|\mathcal{A}^*w\|_{L^s \log^s L(\R^n)}\leq C\|\mathcal{A}v\|_{{\dot W}^{-l}L^s \log^s L(\R^n)}.$$
\end{prop}

\begin{proof}
Since $L^s\log^s L\subset L^s$,
from \cite[Proposition 3.18]{Guerra2019}, we have the (unique) decomposition
$$v=\mathcal{B}u+\mathcal{A}^*w$$
with 
$$\|\mathcal{B}u\|_{\lebe^s(\R^n)}\leq C\|v\|_{\lebe^s(\R^n)},\quad \|\mathcal{A}^*w\|_{\lebe^s(\R^n)}\leq C\|\mathcal{A}v\|_{{\dot W}^{-l,s}(\R^n)}.$$
To improve the estimate on $w$, we recall from the proof of that proposition that
$\mathcal{A}v=\mathcal{A}\mathcal{A}^*w$, and hence that
$$\mathcal{A}^*(\xi)\hat{w}(\xi)=\mathcal{A}^\dagger(\xi)\mathcal{A}(\xi)\mathcal{A}^*(\xi)\hat{w}(\xi)=\mathcal{A}^\dagger\Big(\frac{\xi}{|\xi|}\Big)\frac{\widehat{\mathcal{A}v}(\xi)}{|\xi|^l},$$
where we refer to \cite{Guerra2019} for the precise definition of $\cala^\dagger$, simply noting here that it is a H\"ormander--Mihlin multiplier.
By Theorem \ref{thm:HM} we get the estimate
$$\|\mathcal{A}^*w\|_{\lebe^s\log^s L(\R^n)}\leq C\|\mathcal{A}v\|_{{\dot W}^{-l}L^s\log^s L(\R^n)},$$
as required. 
\end{proof}

The main result of this section is Theorem \ref{thm:hardy_bound} below. Before  proceeding with the proof it is helpful to note that, on a ball $B\equiv B_R(0)$,
$$\Vert f \Vert_{W^{-l} L^p\log^\alpha L(B_R(0))} 
\approx_R 
\Vert f \Vert_{{\dot W}^{-l} L^p\log^\alpha L(B_R(0))}.$$
Indeed, by Poincar\'e's inequality, ${ W}^l L^{p'}\log^{\frac{\alpha}{1-p}} L(B) \cong {\dot W}^l L^{p'}\log^{\frac{\alpha}{1-p}} L(B)$ and, by standard considerations about Sobolev spaces and Theorem \ref{thm:zygmunddual}, we have the identifications
$${W}^{-l} L^p \log^\alpha L(B)\cong 
\left({ W}^l_0 L^{p'}\log^{\frac{\alpha}{1-p}} L(B)\right)^*
\cong
\left({\dot W}^l_0 L^{p'}\log^{\frac{\alpha}{1-p}} L(B)\right)^*
\cong
{\dot W}^{-l} L^p \log^\alpha L(B);
$$
here ${\dot W}^l_0 X$ is defined as the closure of test functions in $\Vert \D^l \cdot \Vert_X$, with a similar definition for the inhomogeneous variant.

\begin{thm}\label{thm:hardy_bound}
Let $F\colon \mathbb{V}\to \R$ be $s$-homogeneous and $\cala$-quasiaffine, where $s\geq 2$.
Then
\begin{align*}
\label{eq:localhardyimp}
	  \begin{rcases}
	  v\in L^s_{\locc}(\R^n,\mathbb{V})\\
	  \cala v\in {\sobo}{_{\locc}^{-l}L^s\log^s L(\R^n,\mathbb{W})}
	  \end{rcases}
	  \implies
	  F(v)\in\mathscr{H}^1_{\locc}(\R^n).
	\end{align*}
In fact, for any $R>0$ we have the estimate
\begin{equation}
    \int_{B_R(0)}\mathcal M_{\locc}[F(v)](x)\dif x\leq C \|\eta v\|_{\lebe^{s}(B_{R+2}(0))} \|\cala (\eta v)\|_{{ \sobo}^{-l} L^s\log^s L},
    \label{eq:localHardy}
\end{equation}
where $\eta\in C^\infty_c(B_{R+2}(0))$ is arbitrary.
\end{thm}

	\begin{proof}
	We use Proposition \ref{prop:HHHH} to write 
	$
	\eta v=\B u +\cala^* w
	$
	 for $u\in \dot{\sobo}{^{k,s}}(\R^n,\mathbb{U})$ and for $w\in \dot{\sobo}{^{l}}L^s\log^s L(\R^n,\mathbb{W})$ such that
\begin{equation}
\label{eq:helmholtzest}
    \begin{split}
         \|\mathcal B u\|_{\lebe^s(\R^n)}& \leq C \|\eta v\|_{\lebe^s(\R^n)}
         ,\\
         \quad\|\mathcal{A}^*w\|_{\lebe^s \log^s L(\R^n)}& \leq C \|\cala(\eta v)\|_{{ \sobo}{^{-l}L^s\log^s L}(\R^n)}
    \end{split}
\end{equation}
Note that both $\mathcal B u,\, \mathcal A^* w$ are supported in $B_{R+2}(0)$ and, furthermore,
\begin{align}\label{eq:exactbound}
\int_{B_R(0)}\mathcal M_{\locc} [F(\B u)](x)\dif x\leq\|F(\B u)\|_{\mathscr{H}^1(\R^n)}\leq C \|\B u\|_{\lebe^s(\R^n)}^s,
\end{align}
where the last estimate follows from Theorem~\ref{thm:CC-standard}.

We now examine 
\begin{align}\label{eq:nonlin_diff}
F(\B u+\cala^*w)-F(\B u)=\sum_{|\alpha|=s}\sum_{\beta<\alpha} c_{\alpha,\beta} (\B u)^{\beta}(\cala^*w)^{\alpha-\beta},
\end{align}
where the multi-indices are taken with respect to an orthonormal basis in $\mathbb{V}$. Write $i\coloneqq|\beta|<s$. Note that $(\B u)^{\beta}\in\lebe^{s/i}(\R^n)$ with an obvious convention if $i=0$. From \eqref{eq:holder}, we deduce
$$
\Vert |\mathcal B u|^i |\cala^* w|^{s-i}\Vert_{L\log L}
\leq C \Vert \mathcal B u\Vert_{L^s}^i \Vert \cala^* w\Vert_{L^s\log^{\frac{s}{s-i}} L}^{s-i}
\leq C_R 
\Vert \mathcal B u\Vert_{L^s}^i \Vert \cala^* w\Vert_{L^s\log^{s} L}^{s-i},
$$
where all norms are taken over $B_{R+2}(0)$.

Since $R$ is arbitrary, we obtain that $|F(\B u+\cala^*w)-F(\B u)|\in\lebe\log\lebe_{\locc}(\R^n)$ which, together with Proposition \ref{prop:stein}, implies that $|F(\B u+\cala^*w)-F(\B u)|\in\mathscr{H}^1_{\locc}(\R^n)$; this, in turn, taking into account \eqref{eq:helmholtzest} and \eqref{eq:exactbound}, implies the statement.
\end{proof}

\begin{rmk}
The above proof admits a simple extension to a more general Orlicz setting as stated in Remark \ref{rmk:sharpen}\ref{itm:sharp}. The crucial point is that the H\"ormander--Mihlin multiplier used in the proof of Proposition \ref{prop:HHHH} is a bounded linear operator between Orlicz spaces satisfying the appropriate assumptions. This allows us to interpolate the assumptions on the sequence and the constraint in order to obtain a sharp statement. See Appendix \ref{app:Orlicz} for details.
\end{rmk}

\section{Theorem \ref{thm:main}: proofs and sharpness}\label{sec:main}
We proceed with the proof of the main result:
\begin{proof}[Proof of Theorem \ref{thm:main}]
Theorem~\ref{thm:main}\ref{itm:M} follows immediately from Theorem~\ref{thm:CC-standard}, which was proved in \cite{Guerra2019}; see also the proof of Theorem~\ref{thm:dist} for a direct proof. 

Assume that we are in the setting of Theorem~\ref{thm:main}\ref{itm:L1}, so that it is elementary to show that $F(v_j)$ is locally equi-integrable. In particular, we have that $(F(v_j))_j$ has a  subsequence weakly convergent in $\lebe^1_{\locc}$, so that we are in a position to apply Theorem~\ref{thm:CDM}. In particular, we obtain the claim along a subsequence, which is enough to conclude since the limit $F(v)$ is independent of the subsequence chosen.

Finally we address Theorem~\ref{thm:main}\ref{itm:H1} by first noting that the case when $L^r \log^\beta L\subseteq L^s\log^s L$ was already dealt with in Theorem~\ref{thm:hardy_bound}. For the case where $\alpha\geq 1$, we claim that $F(v_j)$ is bounded in $\lebe\log\lebe_{\locc}$, which suffices by Proposition \ref{prop:stein}.  We have that
\begin{align*}
\|F(v_j)\|_{L\log L}\leq C\||v_j|^s\|_{L\log L}\leq C\|v_j\|_{L^s\log L}^s
\end{align*}
by \eqref{eq:holder}, which is bounded uniformly in $j$ by assumption.
\end{proof}

The remainder of this section is dedicated to show that Theorem \ref{thm:main} is sharp. For this, we will only consider the classical div-curl case, namely $\cala(v,\tilde v)\equiv (\tp{div\,}v,\tp{curl\,}\tilde v)$ and $F(v,\tilde v)=v\cdot\tilde v$, so that $s=2$. We will work in dimension $n=2$, though it is easy to extend all of the examples below to higher dimensions in a trivial way.

We briefly recall the set up of Theorem~\ref{thm:main}: we assume
$$
(v_j,\tilde v_j)\rightharpoonup (v,\tilde v)\textup{ in }\lebe^s\log ^\alpha L_{\locc}\quad\text{ and }\quad
\cala(v_j,\tilde v_j)\rightarrow\cala (v,\tilde v)\textup{ in }W_{\locc}^{-l}L^r\log^\beta L
$$
and investigate the convergence 
\begin{align}\label{eq:squig_conv}
F(v_j,\tilde v_j)\rightsquigarrow F(v,\tilde v)\text{ in }X,
\end{align}
where $X\in\{\mathscr M_{\locc},\lebe^1_{\locc},\mathscr{H}^1_{\locc}\}$ is an $\lebe^1$-type space endowed with an appropriate weak(-$*$) topology. 

We first establish the sharpness of Theorem~\ref{thm:main}\ref{itm:M}, concerning weak-* convergence of measures. We can only consider the case $\alpha=0$
since otherwise the weak* convergence in the space of measures would hold by Theorem~\ref{thm:main}\ref{itm:L1}. The following example is easily checked:

\begin{ex}\label{ex:a}
Let $r,\beta$ be such that $L^r\log^\beta L_{\locc}\not\subseteq L^2_{\locc}$ 
and let $v_j=\tilde v_j=\mathbbm 1_{(0,j^{-1})^2}{j}e$, for some $e\in \mathbb{S}^1$. We have that
$$
\begin{rcases}
(v_j,\tilde v_j)\rightharpoonup 0&\text{in }\lebe^2\text{ and }\rightarrow0\text{ in }L^r\log^\beta L\\
\cala (v_j,\tilde v_j)\rightarrow0&\text{in }\sobo^{-1}L^r\log^\beta L
\end{rcases}
\quad
\text{ but }
\quad 
\int_{(0,1)^2}F(v_j,\tilde v_j)\dif x=1.
$$
\end{ex}
An example due to Tartar gives optimality of Theorem~\ref{thm:main}\ref{itm:L1} ($X=\lebe^1$, $\rightsquigarrow=\rightharpoonup$):
\begin{ex}[{\cite[Lemma~7.3]{Tartar2010}}]\label{ex:c}
There exists a sequence $(v_j,\tilde v_j)_j$ supported in the unit ball such that
$$
\begin{rcases}
(v_j,\tilde v_j)\rightharpoonup 0&\text{in }\lebe^2\\
\cala(v_j,\tilde v_j)=0
\end{rcases}
\quad \text{ but }\quad \int_{B_{1/2}(0)}F(v_j,\tilde v_j)\dif x\not\rightarrow0.
$$
\end{ex}
We will also construct an example such that $0\leq F(v)\in \lebe^1_{\locc}\setminus L\log L_{\locc}$, which establishes the optimality of Theorem~\ref{thm:main}\ref{itm:H1}:
\begin{ex}\label{ex:d}
Suppose that $L^r\log^\beta L_{\locc}\not\subseteq L^2 \log^2 L_{\locc}$ and $L^r\log^\beta L_{\locc}\subseteq L^1_{\locc}$.
 Then there exists $(v,\tilde v)\in \lebe^2_{\locc}$ such that $\cala(v,\tilde v)\in W^{-l}_{\locc}L^r\log^\beta L$ but $F(v,\tilde v)\notin\mathscr H^1_{\locc}$. 
 
To see this, let $w\in L^r\log^\beta L\setminus L^2\log^2 L(-1,1)$ be a compactly supported function such that 
	$v(x)=(w(x_1),0)\in L^r\log^\beta L\setminus L^2\log^2 L$ and 
	\begin{equation}\label{ass:hat}
	    \sqrt{\psi(|v|)}|v|\log(1+|v|)\not\in \lebe^1,
	\end{equation}
where we write, for simplicity, $\psi(s)\equiv |s|^{r}\log^{\beta}(1+|s|)$. Let 
	$$\tilde v=g(|v|^2)v,\text{ where } g(s^2)=\sqrt{\frac{\psi(s)}{s^2}}.$$
Then we clearly obtain $$\div v(x_1,x_2)=w'(x_1)\in W^{-1}L^2\log^2 L,$$
	but also
	$$\tilde v \in \lebe^2,\quad \tp{curl\,}\tilde v=0.$$
Taking the dot product, we obtain, a.e.\ on $\{\psi(|v|)\geq 1\}$,
	\begin{equation}
	    \begin{aligned}
	    v\cdot \tilde v\log(1+v\cdot \tilde v)=&\,|v|^2g(|v|^2)\log\big(1+|v|^2g(|v|^2)\big)\\
	    =&\,|v|\sqrt{\psi(|v|)}\log\big(1+|v|\sqrt{\psi(|v|)}\big)\\
	    \geq &\, \sqrt{\psi(|v|)}|v|\log(1+|v|),
	    \end{aligned}
	\end{equation}
	which is not in $\lebe^1$ by \eqref{ass:hat}.

\end{ex}

Finally, we note that the examples above can be made more precise. In particular:
\begin{enumerate}
    \item In Example~\ref{ex:a} we see that, without the assumption of Theorem~\ref{thm:main}\ref{itm:M}, the convergence in \eqref{eq:squig_conv} can fail for any choice of $X$;
    \item Else, we can have weak-* convergence in $\mathscr{M}$, but no better. This can be seen by taking $\alpha=\beta=0$ and adding the Examples~\ref{ex:c} and~\ref{ex:d} (after arranging that the resulting sequences have disjoint supports);
    \item We can have weak-$*$ convergence in $\mathscr M$ and $\mathscr{H}^1$ without weak convergence in $\lebe^1$. This follows from Example~\ref{ex:c}, the classical div-curl lemma, and the Hardy bound in \cite[Theorem~II.1]{Coifman1993} (the latter results are generalized by Theorem~\ref{thm:CC-standard});
    \item Finally, we can have weak-$*$ convergence in $\mathscr M$ and weak convergence in $\lebe^1$ without weak-$*$ convergence in $\mathscr{H}^1$. This follows trivially from Example~\ref{ex:d} by taking the constant sequence in which each term is the example given.
\end{enumerate}
We represent these findings graphically in Table \ref{bogdantable}.
\begin{table}[htbp]
\centering
	\begin{tabular}{|c|c|c|c|}
	\hline
	&$\mathscr M$& $\lebe^1$&$\mathscr H^1$\\
	\hline
	(i)&\xmark&\xmark&\xmark\\
	\hline
	(ii)&\cmark&\xmark&\xmark\\
	\hline
	(iii)&\cmark&\xmark&\cmark\\
	\hline
	(iv)&\cmark&\cmark&\xmark\\
	\hline
	\end{tabular}
	    \captionsetup{width=.75\linewidth}
\caption{Whenever one of the assumptions of Theorem~\ref{thm:main} fails, there is also a sequence $(v_j,\tilde v_j)$ such that the convergence \eqref{eq:squig_conv} fails for the corresponding space $X$.}
\label{bogdantable}
\end{table}

\section{Distributional null Lagrangians and Hardy estimates}\label{sec:Hp}

We now begin the proof of Theorem~\ref{thm:dist}.
We recall that we are working with functions defined on a bounded open set $\Omega$, $F\colon \mathbb V\rightarrow \R$ an homogeneous $\cala$-quasiaffine map of degree $s\geq2$, and $ns/(n+1)<q\leq s$, $r\geq r_*$.  Let $(v_j)_j\subset C^\infty_c(\Omega,\mathbb V)$ be such that
        \begin{align*}
           & v_j\rightharpoonup v \text{ in }\lebe^{q}(\Om,\mathbb{V})\\
           & \cala v_j\rightarrow \cala v\text{ in }\sobo^{-l,r}(\Om,\mathbb W).
        \end{align*}
        Then the goal of this section is to prove that
        \begin{equation}
        \label{eq:firstgoal}
        F(v_j)\wstar F(v)\text{ in }\mathscr{D}^\prime (\R^n).
        \end{equation}
        In addition, if  $r>r_*$ or $(\cala v_j)$ is bounded in $W^{-l}L^{r_*}\log^{r_*}L$, we also have the Hardy space bound:
        \begin{equation}
        (F(v_j))_j\text{ is bounded in }\mathscr{H}^{q/s}(\R^n).
        \label{eq:secondgoal}
        \end{equation}
More precisely, we will prove that
$$
\|F(v)\|_{\mathscr H^{\frac{q}{s}}}\leq C\left(\|v\|_{\lebe^q}+\|\cala v\|_{W^{-l}L^{r_*}\log^{r_*}L}\right)^s\quad\textup{for }v\in C_c^\infty(\Omega,\mathbb V).
$$
\begin{proof}[Proof of \eqref{eq:firstgoal}]
We first note that by the restrictions on the parameters, the uniformly bounded support of $v_j$, and the compact Sobolev embedding, we have that
$$
v_j\rightharpoonup v\textup{ in }\lebe^q\implies v_j\rightarrow v\textup{ in }{W}{^{-1,r_*}}.
$$
By Proposition~\ref{prop:HHHH}, we have the decompositions $v_j=\B u_j + \cala^*w_j$ which obey
\begin{align*}
    \begin{cases}\B u_j \rightharpoonup \B u & \textup{ in }\lebe^q\\
    \B u_j \rightarrow \B u& \textup{ in }{W}{^{-1,r_*}}\end{cases} \qquad \tp{and} \qquad 
    \cala^* w_j \rightarrow \cala^* w\textup{ in }\lebe^{r_*}.
\end{align*}
We next fix $\varphi\in C_c^\infty(\Omega)$ and claim that 
\begin{align}\label{eq:cauchy}
\lim_{i,j\rightarrow\infty}\int_{\Omega}\varphi[F(v_j)-F(v_i)]\dif x= 0.
\end{align}
To this end, we will split more carefully than in \eqref{eq:nonlin_diff}:
$$
F(v_j)-F(\B u_j)=F(\B u_j+\cala^*w_j)-F(\B u_j)=\int_0^1 \langle \cala^*w_j,F'(\B u_j+t\cala^*w_j)\rangle\dif t.
$$
First, we use \cite[Equation~(7.2), Remark~7.4]{Guerra2019} to see that
$$
\lim_{i,j\rightarrow\infty}\int_{\Omega}\varphi[F(\B u_j)-F(\B u_i)]\dif x= 0.
$$
We further write, using Fubini's Theorem,
\begin{align*}
    \int_{\Omega}\varphi [F(v_j)-F(\B u_j)]\dif x=\int_0^1\int_{\Omega}\varphi\langle \cala^*w_j,F'(\B u_j+t\cala^*w_j)\rangle\dif x\dif t.
\end{align*}
We first note the simple bound independent of $t\in(0,1)$
$$
|\langle \cala^*w_j,F'(\B u_j+t\cala^*w_j)\rangle|\leq C|\cala^*w_j|(|\B u_j|^{s-1}+|\cala^* w_j|^{s-1}),
$$
where the first term is bounded in $\lebe^{r_*}=(\lebe^{q/(s-1)})^*$, and the latter is bounded in $\lebe^{q/(s-1)}$. Therefore, upon proving that the inner integral converges in $j$, we will be able to apply the dominated convergence theorem to interchange the limit in $j$ with the integral $\dif t$. 
We now recall  that, if $F$ is $\cala$-quasiaffine, then so are the components of $F'$, which are in addition $(s-1)$-homogeneous. Noting that $(F'(\B u_j+t\cala^*w_j))_j$ is  equi-integrable, since it is bounded in $\lebe^{q/(s-1)}$ and $q>s-1$, we can infer from Theorem~\ref{thm:CDM} that
$$
F'(\B u_j+t\cala^*w_j)\rightharpoonup F'(\B u+t\cala^*w)\textup{ in }\lebe^1(\Omega).
$$
This is then automatically improved to weak convergence in $\lebe^{q/(s-1)}$ by the available bound. Recalling that $\cala^* w_j\rightarrow \cala^*w$ in $\lebe^{r_*}$, we can infer that 
$$
\langle \cala^*w_j,F'(\B u_j+t\cala^*w_j)\rangle\wstar \langle \cala^*w,F'(\B u+t\cala^*w)\rangle\textup{ in }\mathscr{D}'(\Omega).
$$
We are thus in a position to indeed apply the dominated convergence theorem to conclude that
\begin{equation}
\lim_{j\rightarrow\infty} \int_{\Omega}\varphi [F(v_j)-F(\B u_j)]\dif x=\int_{\Omega}\varphi \int_0^1\langle \cala^*w,F'(\B u+t\cala^*w)\rangle\dif t\dif x.
\label{eq:stepb}
\end{equation}
This completes the proof of the well defined-ness of the distributional quantity $F(v)$, and (\ref{eq:firstgoal}) follows.
\end{proof}

It remains to prove the Hardy space bound. 

\begin{proof}[Proof of \eqref{eq:secondgoal}]
This proof will be done in two steps, first assuming exact constraint $\cala v_j=0$ and then incorporating the perturbation.

\textbf{Step 1}. We first prove that, if $v\in C_c^\infty(\R^n,\mathbb V)$ satisfies $\cala v=0$, then
\begin{align}\label{eq:Hp_bound}
    \|F(v)\|_{\mathscr H^{q/s}}\leq C\|v\|_{\lebe^q}^{{s}}.
\end{align}
Though our proof of this inequality uses the techniques in \cite{Coifman1993,Guerra2019,Lindberg2017}, the result is probably new. In particular, we adapt the approach used to prove \cite[Proposition~6.3(b)]{Guerra2019}.

We note that it is enough to prove \eqref{eq:Hp_bound} in the case when $v=\D U$ for a test function $U$ and $F=M$ is an $s\times s$ minor. We explain why this is the case:

First, recall from \cite[Proposition~3.16]{Guerra2019} that if $\cala v=0$ in full space, then $v=\B u$ for some $u\in\dot{W}{^{k,q}}(\R^n,\mathbb U)$ with the estimate
$$
\|\D^k u\|_{\lebe^q}\leq C\|v\|_{\lebe^q}.
$$
We can then write in jet notation
$$
\B u=T(\D^k u),
$$
where $T$ is a linear map between finite dimensional spaces (see also \cite[Equation~(3.7)]{Guerra2019}). By \cite[Lemma~5.6]{Guerra2019}, we have that $F$ is $\cala$-quasiaffine if and only if $F\circ T$ is $k$-quasiaffine. These were characterized in \cite[Theorem~4.1]{Ball1981} as linear combinations of Jacobian subdeterminants of $\D U$, where $U\equiv \D^{k-1} u$. Therefore, using the $s$-homogeneity assumption, we have that
$$
F\circ T(\D U)= \sum_{\deg M=s}c_M M(\D U),
$$
where the sum runs over all $s\times s$ minors of the matrix $\D U$. Assuming that \eqref{eq:Hp_bound} holds for $F=M$ and $v=\D U$, we have that
\begin{align*}
    \|F(v)\|_{\mathscr{H}^{\frac{q}{s}}}=\|F\circ T(\D U)\|_{\mathscr{H}^{\frac{q}{s}}}\leq C\sum_{\deg M=s}\|M(\D U)\|_{\mathscr{H}^{\frac{q}{s}}}\leq C\|\D U\|_{\lebe^q}^s = C\|\D^ku\|_{\lebe^q}^s\leq C\|v\|_{\lebe^q}^s.
\end{align*}
Therefore it remains to prove that
\begin{align}\label{eq:minor_inequ}
    \|M(\D U)\|_{\mathscr{H}^{\frac{q}{s}}}\leq C\|\D U\|_{\lebe^q}^s\quad\textup{for }U\in C^\infty_c(\R^n,\mathfrak U),
\end{align}
for an $s\times s$ minor $M$, where we write $\mathfrak U$ for the space of symmetric, $(k-1)$-linear, $\mathbb U$-valued maps on $\R^n$. We introduce coordinates $x=(x',x'')$, $U=(U',U'')$, chosen such that $M(\D U)=\det \D_{x'}U'$, where $\D_{x'}$ is viewed as a differential operator on $\R^n$.

We let $0\not\equiv\psi\in C_c^\infty(B_1(0))$ be non-negative. We recall from the proof of \cite[Proposition~6.3(b)]{Guerra2019} that there exists a vector field $\Sigma$ such that
\begin{align}\label{eq:div-curl_structure}
	M(\D U)=\langle \D_{x'}U'_1, \Sigma \rangle_{\R^s}
	,\quad \D_{x'}^* \Sigma=0,\quad
    |\Sigma|\leq C\prod_{j=2}^s|\D U_j'|,
    \end{align}
where we allow the constants to depend on $\psi$ also. These are just cofactor identities, written in our coordinates. We next estimate
    \begin{align*}    
     |\psi_t*M(\D U)|(x)&=\left|\frac{1}{t^{n}}\int_{\R^n}\langle\D_{x^\prime}U^\prime_1(y),\Sigma(y)\rangle_{\R^s}\psi\left(\frac{x-y}{t}\right)\dif y\right|\\
     &=\left|\fint_{B_t(x)}\biggr\langle\D_{x^\prime}\big[U^\prime_1(y)-(U^\prime_1)_{x,t}\big],\Sigma(y)\psi\left(\frac{x-y}{t}\right)\biggr\rangle_{\R^s}\dif y\right|\\
     &=\left|\frac{1}{t}\fint_{B_t(x)}(U^\prime_1(y)-(U^\prime_1)_{x,t})\biggr\langle (\D_{x^\prime}\psi)\left(\frac{x-y}{t}\right),\Sigma(y)\biggr\rangle_{\R^s}\dif y\right|\\
     &\leq \frac{C}{t}\fint_{B_t(x)}|U_1'-(U_1')_{x,t}||\Sigma|\dif y,
\end{align*}
where we write
$$
\fint_{\Omega}\equiv\frac{1}{\mathscr \lebe^n(\Omega)}\int_{\Omega}\;,\quad (\,\bigcdot\,)_{x,t}\equiv\fint_{B_t(x)}\;,\quad \psi_t(x)\equiv\frac{1}{t^n}\psi\left(\frac{x}{t}\right).
$$
We choose $ns/(n+1)<\tilde q<q$, $p\equiv n\tilde q/(n-\tilde q)$, and $p'=p/(p-1)$ to estimate by use of H\"older and Poincar\'e--Sobolev inequalities
\begin{align*}
    |\psi_t*M(\D U)|(x)&\leq \frac{C}{t}\left(\fint_{B_t(x)}|U_1'-(U_1')_{x,t}|^{p}\dif y\right)^{1/p}\left(\fint_{B_t(x)}|\Sigma|^{p'}\dif y\right)^{1/p'}\\
    &\leq C\left(\fint_{B_t(x)}|\D U_1'|^{\tilde q}\dif y\right)^{1/\tilde q}\left(\fint_{B_t(x)}|\Sigma|^{p'}\dif y\right)^{1/p'}.
\end{align*}
Writing $\mathcal M$ for the Hardy--Littlewood maximal function, we estimate further
\begin{align*}
    \sup_{t>0}|\psi_t*M(\D U)|(x)\leq C\mathcal M(|\D U_1'|^{\tilde q})(x)^{1/\tilde q}\mathcal M(|\Sigma|^{p'})(x)^{1/p'}.
\end{align*}
Finally, we apply H\"older's inequality and the boundedness of the maximal function on Lebesgue spaces to obtain that
\begin{align*}
    \|M(\D U)\|_{\mathscr H^{\frac{q}{s}}}^{\frac{q}{s}}&\leq  C\int_{\R^n} \mathcal M(|\D U_1'|^{\tilde q})^{\frac{q}{s\tilde q}}\mathcal M(|\Sigma|^{p'})^{\frac{q}{sp'}}\dif x\\
    &\leq C\left(\int_{\R^n} \mathcal M(|\D U_1'|^{\tilde q})^{\frac{q}{\tilde q}}\dif x\right)^{1/s}\left(\int_{\R^n} \mathcal M(|\Sigma|^{p'})^{\frac{q}{(s-1)p'}}\dif x\right)^{(s-1)/s}\\
    &\leq C\left(\int_{\R^n} |\D U_1'|^q\dif x\right)^{1/s}\left(\int_{\R^n} |\Sigma|^{\frac{q}{(s-1)}}\dif x\right)^{(s-1)/s}\\
    &\leq C\left(\int_{\R^n} |\D U_1'|^q\dif x\right)^{1/s}\left(\int_{\R^n} \prod_{j=2}^s|\D U_j'|^{\frac{q}{(s-1)}}\dif x\right)^{(s-1)/s}\\
    &\leq C \prod_{j=1}^s\left(\int_{\R^n} |\D U_j'|^{q}\dif x\right)^{1/s}\leq C \|\D U\|_{\lebe^q}^q,
\end{align*}
which completes the proof of \eqref{eq:minor_inequ}.

\textbf{Step 2}.
We will again use the decomposition $v_j=\B u_j+\cala^*w_j$, of which we record:
\begin{align*}
    \|\B u_j\|_{\lebe^q}\leq C\|v_j\|_{\lebe^q},\quad\|\cala^*w_j\|_{\lebe^{r_*}\log^{r_*}L}\leq C\|\cala v_j\|_{{W}^{-l}L^{r_*}\log^{r_*}L},
\end{align*}
by an adaptation of Proposition \ref{prop:HHHH}. 
Since $v_j$ have support in $\Omega$, it suffices to show that $(\mathcal M_{\locc}F(v_j))_j$ is uniformly bounded in $\lebe^{q/s}$. We therefore write $\tilde\Omega\equiv\{x\in\R^n\colon \tp{dist}(x,\Omega)<1\}$ and write
\begin{align*}
\int_{\R^n}|\mathcal{M}_{\locc}F(v_j)|^{q/s}\dif x&= \int_{\tilde \Omega}|\mathcal{M}_{\locc}F(v_j)|^{q/s}\dif x\\
&\leq C \int_{\tilde \Omega}|\mathcal{M}_{\locc}F(\B u_j)|^{q/s}\dif x+C\int_{\tilde \Omega}|\mathcal{M}_{\locc}\left(F(v_j)-F(\B u_j)\right)|^{q/s}\dif x\\
&\leq C \|\B u_j\|_{\lebe^q}^q+C\left(\int_{\tilde \Omega}|\mathcal{M}_{\locc}\left(F(v_j)-F(\B u_j)\right)|\dif x\right)^{q/s},
\end{align*}
where in the last inequality we used \eqref{eq:Hp_bound} and H\"older's inequality. The first term is thus bounded by the $\lebe^q$-norm of $v_j$, 
whereas for the last term we apply Proposition \ref{prop:stein} to estimate $\mathcal{M}_{\locc}$ in $L^1$ by $\|F(v_j)-F(\B u_j)\|_{L\log L}$ and will apply a simple argument to estimate this quantity. To prove the remaining bound, we go back to \eqref{eq:nonlin_diff} and note that we need to control terms of the form
$
(\cala^* w_j)^\alpha (\B u_j)^{\beta} $
with $0<|\alpha|=s-|\beta|$.
We write $|\alpha|=i\in\{1,\ldots,s\}$ and apply the H\"older--Zygmund inequality, \eqref{eq:holder}, to get
\begin{align*}
    \||\cala^* w_j|^i |\B u_j|^{s-i}\|_{L\log L}\leq \|\cala^*w_j\|^i_{L^{r_*}\log^{\frac{r_*}{i}}L}\|\B u_j\|^{s-i}_{\lebe^q}\leq C\|\cala^*w_j\|^i_{L^{r_*}\log^{r_*}L}\|\B u_j\|^{s-i}_{\lebe^q}.
\end{align*}
The conclusion then follows by collecting the considerations above. In fact, we obtain the estimate
$$
\|F(v)\|_{\mathscr H^{\frac{q}{s}}}\leq C\left(\|v\|_{\lebe^q}+\|\cala v\|_{W^{-l}L^{r_*}\log^{r_*}L}\right)^s\quad\textup{for }v\in C_c^\infty(\Omega,\mathbb V),
$$
from the bounds given by the Helmholtz decomposition of Proposition~\ref{prop:HHHH}.
\end{proof}

\begin{rmk}
 It follows from the proof of \eqref{eq:secondgoal} that, at the endpoint $q=\frac{n s}{s+1}$, there is a weak-type estimate
$$\Vert F(v)\Vert_{\mathscr H^{\frac{n}{n+1},\infty}} \leq C(\Omega) \left(\|v\|_{\lebe^\frac{n s}{s+1}}+\|\cala v\|_{W^{-l}L^{r_*}\log^{r_*}L}\right)^s\quad\textup{for }v\in C_c^\infty(\Omega,\mathbb V).$$
In  general, one cannot improve this estimate to one of strong-type: indeed, it suffices to consider the case $\mathcal A=\tp{curl}$ and $F=\det$. The Jacobian has only one cancellation (i.e.\ no higher order moments vanish); however, if $f\in \mathscr H^{\frac{n}{n+1}}(\R^n)$ then $\int_{\R^n} f\dif x = \int_{\R^n} x f(x) \dif x = 0$.
\end{rmk}

\begin{rmk}\label{rmk:CC}
	While proving above the well defined-ness of the distributional quantities in Theorem~\ref{thm:dist}, we also reproved the critical exponent case $q=s=r_*=r$ of \eqref{eq:M_intro}, originally covered in \cite{Guerra2019}. In contrast with the proof there and even with the techniques we were aware of to prove the statement for the supercritical case $q=s=r_*<r$, here we have given a proof that does not rely on semi-continuity methods. We now sketch an alternative self contained method of proof here. First we recall the notation, with
	\begin{align*} 
	v_j\weakto v \text{ in }\lebe^s(\Om,\mathbb V)\quad\text{and}\quad
	\mathcal{A}v_j\to \mathcal{A}v \text{ in } \sobo^{-l,s}(\Om,\mathbb W),
	\end{align*}
	we aim to show that $F(v_j)$ converges to $F(v)$ in the sense of distributions, where $F$ is an $s$-homogeneous $\cala$-quasiaffine polynomial. So we fix $\varphi\in C_c^\infty(\Omega)$ and outline the following steps:
	\begin{enumerate}
		\item We can assume that $v_j\in C_c^\infty(\Omega,\mathbb V)$ by a standard cut-off argument (see, e.g. the proof of \cite[Proposition~2.15]{Fonseca1999}).
		\item We use the Helmholtz decomposition $v_j=\B u_j+\cala^* w_j$ and (\ref{eq:stepb})  to show that
		$$
		F(v_j)-F(\B u_j)\rightharpoonup F(v)-F(\B u)\text{ in }L^1_{\locc}(\Omega).
		$$
		\item It would remain to show that $F(\B u_j)$ converges to $F(\B u)$ in the sense of distributions. By the reduction in Step 1 of the proof of Theorem~\ref{thm:dist}, we can assume that $F=M$ is a minor and $\B u=\D U$, $\B u_j=\D U_j$. We can then use the div-curl structure given in \eqref{eq:div-curl_structure} and an adaptation of the simple proof of the div-curl lemma to conclude. Indeed, with the notation of \eqref{eq:div-curl_structure}, we have, for any fixed $\varphi\in C_c^\infty(\R^n)$,
		\begin{align*}
		\int_\Om \langle \D_{x'}(U_{j})_{1}',\Sigma_j\rangle_{\R^s}\varphi\dif x=-\int_\Om (U_{j})_{1}'\langle \Sigma_j,\D_{x'}\varphi\rangle_{\R^s}\dif x\to &\,-\int_\Om U_{1}'\langle \Sigma,\D_{x'}\varphi\rangle_{\R^s}\dif x\\
		=&\,\int_\Om \langle \D_{x'}U_{1}',\Sigma\rangle_{\R^s}\varphi\dif x,
		\end{align*}
		where we have used that $\D^*_{x}\Sigma_j=\D_{x'}^*\Sigma=0$, the weak convergence $\Sigma_j\weakto \Sigma$ in $L^{\frac{s}{s-1}}$, and the strong convergence (locally) of $(U_{j})_{1}'\to U_1'$ in $L^s$ due to the compact Sobolev embedding.
	\end{enumerate}
\end{rmk}

\section{Quantitative dual H\"older estimates}\label{sec:holder}
In this last section, we prove Theorem \ref{thm:D} and its dual H\"older estimates. As in the proof of Theorem \ref{thm:dist}, the general result follows from the result for distributional Jacobians. As the result for the Jacobian determinant is already new and may be of independent interest, we first state and prove the result in this case in Theorem \ref{thm:interpolated} below. The core of the proof rests on the elegant observation of \cite{Brezis2011b} that for sufficiently smooth functions $u:\R^n\to\R^n$ and $\varphi:\R^n\to\R$, one may write
$$\int_{\R^n} \det(\D u)\varphi\dif x=\int_{\R^{n+1}_+}\det_{n+1}\big(\D_{t,x}\Phi,\D_{t,x} U\big)\dif x\dif t,$$
where $U$, $\Phi$ are extensions of $u$ and $\varphi$ to the upper half space. Throughout this section, whenever we refer to the harmonic extension of a function defined on $\R^n$, we mean the extension to the half-space $\R^{n+1}_+$ through convolution with the Poisson kernel.

We begin by recalling a useful fact concerning properties of the harmonic extension.

\begin{prop}\label{prop:fractional}
Let $f\in C^\infty_c(\R^n)$ and denote by $F(t,x)$ the harmonic extension of $f$ to $\R^{n+1}_+$. Then, for any $\beta\in[0,1)$, $p\in(1,\infty)$, we have that
\begin{equation}\label{eq:fracSobest}
    \Big(\int_{\R^{n+1}_+}\big|t^{1-\frac{1}{p}-\beta}\D_{t,x}F(t,x)\big|^p\,\dif x\dif t\Big)^{\frac{1}{p}}\leq C [f]_{W^{\beta,p}(\R^n)},
\end{equation}
where in the case $\beta=0$, the semi-norm on the right is the $L^p$ norm.
Let $\alpha\in(0,1)$, $\varphi\in C^{0,\alpha}(\R^n)$ and denote by $\Phi(t,x)$ the harmonic extension of $\varphi$ to $\R^{n+1}_+$. Then we have
\begin{equation}\label{eq:Holderest}
    \sup_{t,x}t^{1-\alpha}\big|\D_{t,x} \Phi(t,x)\big|\leq C [\varphi]_{C^{0,\alpha}}.
\end{equation}
\end{prop}
Inequality \eqref{eq:fracSobest} in the case $\beta\neq 0$ may be found in \cite[Proposition 10.2, (10.7), (10.9)]{Lenzmann2018}. In the case $\beta=0$, this is a standard estimate, but may be seen also, for example, in \cite[Theorem 10.8]{Lenzmann2018}, recalling that the Besov space $B^0_{p,p}=L^p$ for $p\in(1,\infty)$. Inequality \eqref{eq:Holderest} is the estimate of \cite[Theorem 10.6]{Lenzmann2018}.

With the help of Proposition \ref{prop:fractional}, we can prove the following result:

\begin{thm}\label{thm:interpolated}
Let $\beta\in[\frac{n-1}{n},1)$, $\alpha\in(0,1]$ such that $\frac{\alpha}{n}=1-\beta$. Then there exists $C>0$ such that, for any $u\in W^{\beta,n}(\R^n;\R^n)$ and $\varphi\in C^{0,\alpha}(\R^n)$, the following estimate holds:
\begin{equation}
    \Big|\int_{\R^n}\det(\D u)\varphi\dif x\Big|\leq C[\varphi]_{C^{0,\alpha}}\prod_{i=1}^n[u_i]_{W^{\beta,n}(\R^n)}.
\end{equation}
Moreover, given another function $v\in W^{\beta,n}(\R^n;\R^n)$, we may estimate the difference of the Jacobians of $u$ and $v$ by 
\begin{equation}
    \Big|\int_{\R^n}\big(\det(\D u)-\det(\D v)\big)\varphi\dif x\Big|\leq C[\varphi]_{C^{0,\alpha}}\sum_{j=1}^n\Big([u_j-v_j]_{W^{\beta,n}(\R^n)}\prod_{i=1}^{j-1}[v_i]_{W^{\beta,n}(\R^n)}\prod_{i=j+1}^n[u_i]_{W^{\beta,n}(\R^n)}\Big).
\end{equation}
\end{thm}

\begin{proof}
The case $\al=1$, $\be=\frac{n}{n-1}$ is proved in \cite[Theorem 3]{Brezis2011b}. We therefore take $\al\in(0,1)$, $\be\in(\frac{n}{n-1},1)$ and first assume that $v\equiv 0$. Let $u$, $\varphi$ be as in the theorem and denote by $U$ and $\Phi$ their harmonic extensions by convolution with the Poisson kernel. Then we have the key identity (see \cite{Brezis2011b,Lenzmann2018})
\begin{align*}
    \int_{\R^n}\det(\D u)\varphi\dif x=&\,\int_{\R^{n+1}_+}\det_{n+1}\big(\D_{t,x}\Phi,\D_{t,x} U\big)\dif x\dif t,
\end{align*}
which is a direct consequence of integration by parts, together with the decay of $U$ and $\Phi$ at infinity. Hence we can estimate
\begin{align*}
    \Big| \int_{\R^n}\det(\D u)\varphi\dif x\Big|\leq&\,\int_{\R^{n+1}_+}\Big(\prod_{i=1}^n|\D U_i|\Big)|\D\Phi|\dif x\dif t\\
    \leq&\, \|t^{1-\alpha}|\D_{t,x}\Phi|\|_{L^\infty_{t,x}}\int_{\R^{n+1}_+}t^{\alpha-1}\Big(\prod_{i=1}^n|\D U_i|\Big)\dif x\dif t\\
    \leq&\, C[\varphi]_{C^{0,\alpha}}\int_{\R^{n+1}_+}\Big(\prod_{i=1}^nt^{\frac{\alpha-1}{n}}|\D U_i|\Big)\dif x\dif t\\
    \leq&\, C[\varphi]_{C^{0,\alpha}}\prod_{i=1}^n\Big(\int_{\R^{n+1}_+}\big|t^{\frac{\alpha-1}{n}}|\D U_i|\big|^n\dif x\dif t\Big)^{\frac{1}{n}}\\
    \leq &\,C[\varphi]_{C^{0,\alpha}}\prod_{i=1}^n[u_i]_{W^{\beta,n}(\R^n)},
\end{align*}
where we have used in the last line that $\frac{\alpha-1}{n}=1-\frac{1}{n}-\beta$ and applied \eqref{eq:fracSobest}.

To address the case with a difference of functions, we recall the standard fact (see, for example, \cite{Brezis2011b}), that we may write
\begin{equation}\label{eq:BNdetdiff}
\det(\D u)-\det(\D v)=\sum_{j=1}^n \textrm{W}^{(j)},
\end{equation}
where
$$\textrm{W}^{(j)}=\det(\D v_1,\ldots,\D v_{j-1},\D(u_j-v_j),\D u_{j+1},\ldots,\D u_n).$$
Writing
$$w^{(j)}=( v_1,\ldots, v_{j-1},u_j-v_j, u_{j+1},\ldots, u_n),$$
so that $\textrm{W}^{(j)}=\det(\D w^{(j)})$, we follow the argument above, replacing $u$ with $w^{(j)}$, summing over $j$ to conclude.
\end{proof}

\begin{cor}\label{cor:GNinterpolation}
Let $\alpha\in(0,1]$ and choose $p\in(n-1,\infty)$, $q\in(1,\infty)$ such that $\frac{\alpha}{q}+\frac{n-\alpha}{p}=1$. Then
\begin{equation}
    \Big|\int_{\R^n}\det(\D u)\varphi\dif x\Big|\leq C[\varphi]_{C^{0,\alpha}} \|u\|^\alpha_{L^q}\|\D u\|_{L^p}^{n-\alpha}.
\end{equation}
If instead we determine $p$, $q$ by $\frac{1}{q}+\frac{n-1}{p}=1$, then we have
\begin{equation}
    \Big|\int_{\R^n}\det(\D u)\varphi\dif x\Big|\leq C[\varphi]_{C^{0,\alpha}} [u]_{W^{1-\alpha,q}}\|\D u\|_{L^p}^{n-1}.
\end{equation}
\end{cor}

\begin{proof}
We recall the following Gagliardo--Nirenberg interpolation inequality from \cite[Corollary 3.2]{Brezis2001}:
For $\th\in(0,1)$, $0\leq \beta_1<\beta_2<\infty$, $p_1,p_2\in(1,\infty)$, determining $\beta=\theta \beta_1+(1-\theta) \beta_2$ and $\frac{1}{r}=\frac{\theta}{p_1}+\frac{1-\theta}{p_2}$, we have 
$$\|u\|_{W^{\beta,r}}\leq C \|u\|_{W^{\beta_1,p_1}}^\theta \|u\|_{W^{\beta_2,p_2}}^{1-\theta}.$$
The first claimed inequality now follows from the choice $\beta$ as in Theorem \ref{thm:interpolated}, that is, $\frac{\alpha}{n}=1-\beta$. We then choose $\beta_1=0$, $\beta_2=1$ and $r=n$. Then $\theta=1-\beta=\frac{\alpha}{n}$ and we conclude that the claimed relation for $p$ and $q$ implies that 
$$ \Big|\int_{\R^n}\det(\D u)\varphi\dif x\Big|\leq C[\varphi]_{C^{0,\alpha}} \|u\|_{W^{\beta,n}}^n\leq C[\varphi]_{C^{0,\alpha}} \|u\|^\alpha_{L^q}\|\D u\|_{L^p}^{n-\alpha}.$$
For the second inequality, we again take $\beta$ as above, but now use $\theta=\frac{1}{n}$, so that $\beta_1=1-\alpha$, $\beta_2=1$, $r=n$, and derive the claimed relation between $p$ and $q$.

To replace the norms with semi-norms, we make the usual observation that subtracting constants from the components of $u$ does not change the Jacobian determinant, and hence we may eliminate the zeroth-order contributions to the norms. 
\end{proof}

Similar statements hold with differences, though note that using the method of Gagliardo--Nirenberg interpolation (as in the above proof) means spreading the norms evenly across all components. See Corollary \ref{cor:D'} for a precise statement.

\begin{rmk}
Anisotropic versions of these inequalities are also available. Let $\alpha\in(0,1)$ and $\beta_i\in(0,1)$ for each $i$ satisfy $\sum_{i=1}^n\beta_i=n-\alpha$, $p_i\in(1,\infty)$ such that $\sum_{i=1}^n\frac{1}{p_i}=1$. Then we have the bound
$$\Big|\int_{\R^n}\det(\D u)\varphi\dif x\Big|\leq C[\varphi]_{C^{0,\alpha}}\prod_{i=1}^n[u_i]_{W^{\beta_i,p_i}}.$$
The proof is as in Theorem \ref{thm:interpolated}, using \eqref{eq:fracSobest}. 
\end{rmk}

Finally, we prove Theorem \ref{thm:D} by reducing to the situation in Theorem \ref{thm:interpolated} and adapting the proof.

\begin{proof}[Proof of Theorem \ref{thm:D}]
As in the proof of Theorem \ref{thm:dist}, we may make the usual reduction to the case of a Jacobian subdeterminant by using the Helmholtz decomposition to write $v=\B \tilde u$ and then $F(\B \tilde u)=F\circ T(\D u)=\sum_{\deg M=s}c_M M(\D u)$ with $u=\D^{k-1}\tilde u$.

We therefore begin by considering first subdeterminants of $\D u$ for $u\in C_c^\infty(\R^n)$ and then extend the estimate by density. To handle the case in which $\Omega$ is a bounded Lipschitz domain, we recall that such domains are extension domains for general Besov spaces (including fractional Sobolev spaces) so that we may extend $u\in W^{\beta,s}(\Omega)$ to a compactly supported function $\bar u\in W^{\be,s}(\R^n)$ with $\|\bar u\|_{W^{\beta,s}(\R^n)}\leq C(\Omega)\|u\|_{W^{\beta,s}(\Omega)}$ and apply the estimate in $\R^n$ proved below. To replace the norms with semi-norms, we simply apply the standard argument that subtracting constants from the function $u$ does not affect its Jacobian determinant.

To address the subdeterminants, we make a simple modification of the proof of Theorem \ref{thm:interpolated} in order to apply the extension identity of Brezis--Nguyen, \cite{Brezis2011b}. Rather than employing the notation of \eqref{eq:div-curl_structure}, we  reorder the coordinates and the rows of the original matrix $\D u$, so that we may assume without loss of generality that we are working with the first principal $s$-minor, i.e. the subdeterminant
$$\det_s(\D_s u_{1},\ldots,\D_s u_{s}),$$
where $\D_s=\begin{pmatrix}\partial_1,
\ldots,
\partial_s\end{pmatrix}^\top$. We then observe the trivial fact that
\begin{align*}
\det_s(\D_s u_{1},\ldots,\D_s u_{s})=&\,\pm\det_n(\D_s u_{1},\ldots,\D_s u_{s},e_{s+1},\ldots,e_n)\\
=&\,\pm\det_n(\D u_{1},\ldots,\D u_{s},\D (\rho(x)x_{s+1}),\ldots,\D (\rho(x)x_n)),
\end{align*}
where $e_j$ is the standard basis vector, $x_j$ is the coordinate function, and we choose $R>0$ such that $\supp\,u\subset B_R(0)$ and then take $\rho\in C^\infty_c(B_{2R}(0))$ such that $0\leq\rho\leq 1$, $\rho=1$ on $B_R(0)$, and $|\D\rho|\leq\frac{2}{R}$.

We treat the cases $\al\in(0,1)$ and $\al=1$ separately. First take $\al\in(0,1)$. We then make the usual harmonic extensions of $u$ and $\varphi$ to $U$ and $\Phi$ by convolution with the Poisson kernel. To extend each $\rho(x)x_j$, $j=s+1,\ldots,n$, we multiply with a function $\phi\in C^\infty_c((-1,\infty))$ such that $\phi(0)=1$ and $|\phi'(t)|\leq \frac{1}{R}$ for all $t\geq 0$, to obtain
\begin{align*}
\Big|\int_{\R^n}&\,\det_s(\D_s u_{1},\ldots,\D_s u_{s})\varphi\dif x\Big|\\
=&\,\Big|\int_{\R^{n+1}_+}\det_{n+1}\big(\D_{t,x}\Phi,\D_{t,x} U_1,\ldots,\D_{t,x} U_s,\D_{t,x} (\rho(x)x_{s+1}\phi(t)),\ldots,\D_{t,x} (\rho(x)x_n\phi(t))\big)\dif x\dif t\Big|\\
\leq&\,C[\varphi]_{C^{0,\alpha}}\int_{\R^{n+1}_+}t^{\al-1}|\D_{t,x} U_1|\dots|\D_{t,x} U_s|\,\dif x\dif t\\
\leq&\,C[\varphi]_{C^{0,\alpha}}\prod_{j=1}^s\Big(\int_{\R^{n+1}_+}\big(t^{\frac{\al-1}{s}}|\D_{t,x} U_j|\big)^s\dif x\dif t\Big)^{\frac{1}{s}}\\
\leq&\,C[\varphi]_{C^{0,\alpha}}[u]_{W^{\beta,s}}^s,
\end{align*}
as required, where we have used that $\frac{\alpha}{s}=1-\beta$ and \eqref{eq:fracSobest} in the last line, and that the functions $$\D_{t,x}(\rho(x)x_j\phi(t))=(\rho(x)+x_j\partial_{x_j}\rho)\phi(t)e_{j}+\sum_{\substack{i=1\\i\neq j}}^n\partial_{x_i}\rho x_j\phi(t)e_i+\rho(x)x_j\phi'(t)e_{n+1}$$ for $j=s+1,\ldots,n$ are all uniformly bounded due to the bounds on $\D\rho$ and $\phi'$.

In the case $\al=1$, the estimate \eqref{eq:Holderest} fails, and so we instead employ the fact that $W^{\frac{s}{s-1},s}$ is the trace space of $W^{1,s}$. We take the extension of $u$ by averages to $U:(0,1)\times\R^n \to\R^n$:
\begin{equation*}
U(t,x)=\fint_{B_t(x)}u(y)\dif y,\text{ so that } \|\D U\|_{L^s((0,1)\times\R^n)}\leq C\|u\|_{W^{\frac{s}{s-1}}(\R^n)}
\end{equation*}
by standard trace theory. We extend $\varphi$ to $\Phi\in C^{0,1}_c([0,1)\times\R^n)$ such that $\|\D\Phi\|_{L^\infty([0,1)\times\R^n)}\leq C\|\D\varphi\|_{L^\infty(\R^n)}$ and extend the coordinate functions as above. Then we again have the extension identity (note that the compact support of $\Phi$ with respect to $t$ ensures that no other boundary term appears) and estimate
\begin{align*}
\Big|\int_{\R^n}&\,\det_s(\D_s u_{1},\ldots,\D_s u_{s})\varphi\dif x\Big|\\
=&\,\Big|\int_{(0,1)\times \R^{n}}\det_{n+1}\big(\D_{t,x}\Phi,\D_{t,x} U_1,\ldots,\D_{t,x} U_s,\D_{t,x} (\rho(x)x_{s+1}\phi(t)),\ldots,\D_{t,x} (\rho(x)x_n\phi(t))\big)\dif x\dif t\Big|\\
\leq&\,C\|\D\Phi\|_{L^\infty([0,1)\times\R^n)}\|\D U\|^s_{L^s((0,1)\times\R^n)}\\
\leq&\, C[\varphi]_{C^{0,1}(\R^n)}\|u\|_{W^{\frac{s}{s-1},s}(\R^n)}^s,
\end{align*}
and we replace the norms with semi-norms by the usual considerations.

Returning now to our original function $v$, we note that, by the Helmholtz decomposition,
$$[u]_{W^{\beta,s}}=[\D^{k-1} \tilde u]_{W^{\beta,s}}\leq C[v]_{W^{\beta-1,s}},$$
which is justified as follows: First, recall from the Helmholtz decomposition that 
$$
\mathcal{F}\tilde u(\xi)=\B^\dagger(\xi)\hat v(\xi),\text{ so }\hat u(\xi)=\mathcal{F}(\D^{k-1}\tilde u)(\xi)=\B^\dagger(\xi)\hat v(\xi)\otimes \xi^{\otimes (k-1)}=\B^\dagger(\xi)\frac{\hat v(\xi)}{|\xi|}\otimes \xi^{\otimes (k-1)}|\xi|
$$
(we again refer to \cite{Guerra2019} for the notation $\B^\dagger$), so that \cite[Theorem~5.2.2 and~5.2.3.1(i)]{Triebel1983} imply that
$$
[u]_{W^{\beta,s}}\leq C\left[\mathcal{F}^{-1}\left(\frac{\hat{v}(\xi)}{|\xi|}\right)\right]_{W^{\beta,s}}\leq C[v]_{W^{\beta-1,s}}.
$$
Finally, to obtain the statement in the case of a difference of functions, we make the same reduction to the case of subdeterminants, recall \eqref{eq:BNdetdiff} and estimate each term on the right as above.
\end{proof}

The following corollary is deduced from Theorem \ref{thm:D} by the Gagliardo--Nirenberg interpolation theorem, following exactly the same argument as that used in Corollary \ref{cor:GNinterpolation}.

\begin{cor}
\label{cor:D'}
Let $\Omega\subset\R^n$ be either a bounded Lipschitz domain  or $\Omega=\R^n$, $F\colon \mathbb V\rightarrow \R$ be an homogeneous $\cala$-quasiaffine map of degree $s\geq2$. Suppose that $u,v\in C^\infty_c(\Omega,\mathbb V)$ are $\cala$-free. Let $\alpha\in (0,1]$ and let $p\in (s-1,\infty)$ and $q\in (1,\infty)$ be such that $\frac{1}{q}+\frac{s-1}{p}=1$. Then, for any $\varphi\in C^{0,\alpha}(\Omega)$, we obtain the estimates
\begin{align*}
\Big|\int_{\Omega}&\, \big(F(u)-F(v)\big)\varphi\,\dif x\Big|\\
&\leq C[\varphi]_{C^{0,\alpha}} [u-v]^{\frac{1}{s}}_{W^{-\alpha,q}}\| u-v\|_{L^p}^{1-\frac{1}{s}}\big([u]_{W^{-\alpha,q}}+[v]_{W^{-\alpha,q}}\big)^{\frac{s-1}{s}}\big(\|u\|_{L^p}+\|v\|_{L^p}\big)^{\frac{(s-1)^2}{s}}.
\end{align*}  
\end{cor}
In the case that $v\equiv 0$, $\alpha=1$, this returns precisely the estimates of \cite[Theorem 1]{Brezis2011b}, \cite[Proposition 7.1]{Guerra2019}.

We conclude this section by proving that Theorem \ref{thm:D} is optimal on the scale of fractional Sobolev spaces. More precisely, we work with the Jacobian determinant and prove the following proposition.
\begin{prop}\label{prop:sharpjacobian}
Let $\Omega\subset\R^n$, $n\geq 2$, be an open, bounded domain, $\al\in(0,1)$, and suppose $\beta\in(0,1)$, $p\in(1,\infty)$ are such that $W^{\beta,p}\not\hookrightarrow W^{\frac{n-\alpha}{n},n}$. Then there exist sequences $(u^{(k)})_k\subset C^1(\Omega,\R^n)$, $(\varphi^{(k)})_k\subset C^1_c(\Omega)$ such that
$$\|u^{(k)}\|_{W^{\be,p}},\:\|\varphi^{(k)}\|_{C^{0,\alpha}}\quad\text{ are uniformly bounded,}$$
and also
$$\int_{\Omega}\det(\D u^{(k)})\varphi^{(k)}\dif x\to\infty\text{ as }k\to\infty.$$
\end{prop}

\begin{proof}
We begin by distinguishing three cases for which the embedding $W^{\beta,p}\hookrightarrow W^{\frac{n-\alpha}{n},n}$ fails.\\
Case 1: $p\leq n$ and $\beta+\frac{\alpha}{n}<\frac{n}{p}$.\\
Case 2: $p>n$ and $\beta+\frac{\alpha}{n}<1$.\\
Case 3: $p>n$ and $\beta=1-\frac{\alpha}{n}$.

\textbf{Case 1:} The argument for Case 1 is simple, and follows almost immediately from the construction in \cite{Brezis2011b}. We recall from \cite[Proof of Remark 1]{Brezis2011b} that there exists a function $g\in C_c^\infty(B_1(0))$ such that the function $g_\eps(x)=\eps^{-\frac{1}{n}}g(\frac{x}{\eps})$ for $\eps>0$ satisfies that for all $\varphi\in C^1(B_1(0))$, 
$$\int_{B_1}\det(\D g_\eps)\varphi\dif x=\sum_{j=1}^n a_j\frac{\partial\varphi}{\partial x_j}(0)+O(\eps\|\D\varphi\|_{L^\infty}),$$
where $a_1>0$ and, moreover, that for any $\beta\in(0,1)$, $p\in(1,\infty)$, we have
$$[g_\eps]^p_{W^{\beta,p}}\approx \eps^{n-\beta p-\frac{p}{n}}.$$
We choose a sequence of test functions $\varphi^{(k)}\in C_c^\infty(B_1)$, uniformly bounded in $C^{0,\al}$, such that $$\frac{\partial\varphi^{(k)}}{\partial x_j}(0) =\begin{cases}
\big(\frac{1}{k}\big)^{\al-1} &\text{ if }j=1,\\
0 & \text{ if }j=2,\ldots n
\end{cases}$$ 
and also $\|\D\varphi^{(k)}\|_{L^\infty}\leq Ck^{1-\alpha}$ where $C$ is independent of $k$.
Such a sequence can easily be achieved by taking a mollification at order $1/k$ of the function
$\varphi(x)=\varphi_1(x_1)\tilde\varphi(x_2,\ldots,x_n)$, where $$\varphi_1(x_1)=\begin{cases}
x_1^\al & x\geq 0,\\
0, & x<0
\end{cases}$$ and $\tilde\varphi\in C_c^\infty(B_1(0))$ in $\R^{n-1}$ with $\tilde\varphi\equiv 1 \text{ on }B_{1/2}(0)$. 

We then define our sequence $u^{(k)}$ as $$u^{(k)}(x)=k^{\frac{\alpha-1}{n}+\gamma}g_{1/k}(x),$$
where $\gamma>0$ is chosen such that $\gamma<\frac{n}{p}-\beta-\frac{\alpha}{n}$. Then we have the fractional Sobolev estimate
$$\|u^{(k)}\|_{W^{\beta,p}}\approx k^{-\frac{1-\alpha}{n}+\gamma-\frac{n}{p}+\beta+\frac{1}{n}}\to 0\text{ as }k\to\infty,$$
by definition of $\gamma$. 

Finally, we note that
\begin{align*}
\int_{B_1}\det(\D u^{(k)})\varphi^{(k)}\dif x=&\, k^{\alpha-1+n\gamma}\big(a_1\frac{\partial\varphi^{(k)}}{\partial x_1}(0)+O(k^{-\alpha})\big)\\
=&\, a_1k^{\alpha-1+n\gamma+1-\alpha}+k^{\alpha-1+n\gamma}O(k^{-\alpha})\\
=&\, a_1k^{n\gamma}+O(k^{-1+n\gamma}),
\end{align*}
which tends to $\infty$ as $k\to\infty$ as $\gamma>0$, $\alpha<1$.

The constructions for Cases 2 and 3 share many similarities. As the example we use in Case 2 is much simpler than that for Case 3 (and the proof is correspondingly much shorter), we choose to include the proof of Case 2 here to aid the reader's comprehension. The basic idea is to construct a sequence of oscillating terms (or, in Case 3, sums of oscillating terms) at increasingly high frequencies, weighted by frequency dependent factors to ensure the uniform boundedness in the fractional Sobolev norm. By choosing appropriate oscillating functions, we ensure that after integration by parts in the Jacobian determinant, we obtain a function whose integral is bounded below by a constant, while the frequency dependent weights tend to infinity. In Case 3, we must choose the frequencies over which we sum to be sufficiently sparse to ensure the boundedness of the cross-terms in the determinant product. Finally, unlike in \cite{Brezis2011b}, it is not sufficient to choose a single test function, but rather the sequence $\varphi^{(k)}$ is taken as an oscillatory sequence at increasing frequencies, with the oscillating factor chosen to complement the oscillations of the sequence $u^{(k)}$.

\textbf{Case 2.} We assume without loss of generality that the ball $B_2(0)\subset \Omega$. Let $\beta_1\in(\beta,\frac{n-\alpha}{n})$ and define, for $k\in\mathbb N$,
\begin{align*}
    u_i^{(k)}=&\,k^{-\beta_1}\sin(kx_i)\qquad\text{ for }i=1,\ldots,n-1,\\
    u_n^{(k)}=&\,-k^{-\beta_1}\cos(kx_n)\phi(x),\quad \phi\in C_c^\infty(B_2),\:\phi\equiv 1\text{ on }B_1,\:0\leq\phi\leq1,\\
    \varphi^{(k)}=&\,k^{-\alpha}\sin(kx_n)\prod_{i=1}^{n-1}\cos(kx_i).
\end{align*}
Throughout the proofs of Cases 2 and 3, constants $C$ and $c$ will always be independent of $k$.
Then we have the usual identity by integration by parts,
\begin{align*}
    \int_{\Omega}&\,\det(\D u^{(k)})\varphi^{(k)}\dif x=-\int_{\Omega}u_n^{(k)}\det(\D u_1^{(k)},\ldots,\D u_{n-1}^{(k)},\D \varphi^{(k)})\dif x.
\end{align*}
A simple computation gives that the determinant on the right is 
$$\det(\D u_1^{(k)},\ldots,\D u_{n-1}^{(k)},\D \varphi^{(k)})=k^{(n-1)(1-\beta_1)+1-\alpha}\cos(kx_n)\prod_{i=1}^{n-1}\cos^2(kx_i).$$
Thus, as $B_1(0)\subset \Omega$, there exists $c>0$ such that the integral is bounded below by
\begin{align*}
    \int_{\Omega}&\,\det(\D u^{(k)})\varphi^{(k)}\dif x=\int_{\Omega}\phi(x)k^{(n-1)(1-\beta_1)+1-\alpha-\beta_1}\cos^2(kx_n)\prod_{i=1}^{n-1}\cos^2(kx_i)\dif x\\
    \geq  &\, ck^{(n-1)(1-\beta_1)+1-\alpha-\beta_1}\\
    =&\, ck^{n-\alpha-n\beta_1},
\end{align*}
which tends to $\infty$ as $k\to\infty$ by construction of $\beta_1$. To check that $u^{(k)}$ is uniformly bounded in $W^{\beta,p}$, we note that
$$\|u^{(k)}\|_{L^\infty}\leq C k^{-\beta_1},\quad \|\D u^{(k)}\|_{L^\infty}\leq C k^{1-\beta_1},$$
hence, by interpolation,
$$[u^{(k)}]_{C^{0,\beta_1}}\leq C.$$
Applying the embedding of $C^{0,\beta_1}\hookrightarrow W^{\beta,p}$ (as $\beta<\beta_1$), we get
$$\|u^{(k)}\|_{W^{s,p}}\leq C.$$
The uniform estimate 
$$[\varphi^{(k)}]_{C^{0,\alpha}}\leq C$$
follows similarly and so we easily conclude Case 2.

\textbf{Case 3.} For Case 3, we begin with the following preliminary notation:\\
For each $k\gg1$, let $n_\ell=k^{\frac{n^2}{\alpha}}8^\ell$ for $\ell=1,\ldots,k$.
Then, trivially, we have the basic estimates 
\begin{equation}\label{ineq:nell}
n_{\ell+1}\geq 4n_\ell \quad\text{ for all }\ell=1,\ldots k-1,
\end{equation}
 and 
\begin{equation}\label{ineq:min_n_diff}
\min_{i\neq j}|n_{\ell_i}-n_{\ell_j}|\geq k^{\frac{n^2}{\alpha}}.
\end{equation}
Then set
\begin{align*}
    u_i^{(k)}=&\,\sum_{\ell=1}^k\frac{1}{n_{\ell}^{\frac{n-\alpha}{n}}(\ell+1)^{\frac{1}{n}}}\sin(n_\ell x_i)\quad\text{ for }i=1,\ldots,n-1,\\
    u_n^{(k)}=&\,-\sum_{\ell=1}^k\frac{1}{n_{\ell}^{\frac{n-\alpha}{n}}(\ell+1)^{\frac{1}{n}}}\cos(n_\ell x_n)\phi(x),\quad \phi\in C_c^\infty(B_2),\:\phi\equiv 1\text{ on }B_1,\:0\leq\phi\leq1,\\
    \varphi^{(k)}=&\,\sum_{\ell=1}^k\frac{1}{n_\ell^{\alpha}}\sin(n_\ell x_n)\prod_{i=1}^{n-1}\cos(n_\ell x_i).
\end{align*}
First we need to check that $u^{(k)}$ is uniformly bounded in $W^{\frac{n-\alpha}{n},n}$ and $\varphi^{(k)}$ is uniformly bounded in $C^{0,\alpha}$. The first of these claims follows similarly to the proof of equation (3.16) in \cite{Brezis2011b}, where we simply note that the change in fractional order $\frac{n-\alpha}{n}$ corresponds to the adjustment we have made to the exponent of $n_\ell$ in the definition of $u^{(k)}$.

Checking $\varphi^{(k)}$ is uniformly bounded in $C^{0,\alpha}$ is similarly straightforward. Indeed, the Littlewood-Paley projection of $\varphi^{(k)}$ at order $j$, $P_j\varphi^{(k)}$ is 
$$P_j\varphi^{(k)}=\sum_{\substack{\ell=1\\ n_\ell\in (2^{j-1},2^{j+1})}}^k\frac{1}{n_\ell^\alpha}\sin(n_\ell x_n)\prod_{i=1}^{n-1}\cos(n_\ell x_i).$$
As there is at most one value of $n_\ell$, which we call $n_{\ell_j}$, for each $j\in\mathbb{Z}$ such that $n_{\ell_j}\in(2^{j-1},2^{j+1})$ by construction of $n_\ell$, this is clearly bounded in $L^\infty$ as
$$\|P_j\varphi^{(k)}\|_{L^\infty}\leq C n_{\ell_j}^{-\alpha},\text{ where }n_{\ell_j}\in (2^{j-1},2^{j+1}).$$
Thus we obtain
$$\|\varphi^{(k)}\|_{C^{0,\alpha}}\leq C\big\|2^{\alpha j}\|P_j\varphi^{(k)}\|_{L^\infty}\big\|_{\ell^\infty}\leq C,$$
where we have used the standard identification of the Besov space $B^{\alpha}_{\infty,\infty}\cong C^{0,\alpha}$ as $\al\in(0,1)$.

It therefore remains only to check that the integral of the Jacobian of $u^{(k)}$ tested against $\varphi^{(k)}$ converges to infinity. We will show that this integral has a lower bound that grows logarithmicly in $k$.
First, we observe
$$\partial_n\varphi^{(k)}(x)=\sum_{\ell=1}^k n_\ell^{1-\alpha}\cos(n_\ell x_n)\prod_{i=1}^{n-1}\cos(n_\ell x_i).$$
Thus a simple calculation shows
\begin{align*}
    \det(&\,\D u_{1}^{(k)},\ldots,\D u_{n-1}^{(k)},\D \varphi^{(k)})\\
    =&\Big(\prod_{i=1}^{n-1}\sum_{\ell=1}^k\frac{n_{\ell_i}^{\frac{\alpha}{n}}}{(\ell_i+1)^{\frac{1}{n}}}\cos(n_{\ell_i}x_i)\Big) \Big(\sum_{\ell_{n+1}=1}^kn_{\ell_{n+1}}^{1-\alpha}\cos(n_{\ell_{n+1}}x_n)\prod_{j=1}^{n-1}\cos(n_{\ell_{n+1}}x_j)\Big)\\
    =&\sum_{\ell=1}^k\frac{1}{(\ell+1)^{\frac{n-1}{n}}}n_{\ell}^{\alpha\frac{n-1}{n}+1-\alpha}\cos(n_\ell x_{n})\prod_{i=1}^{n-1}\cos^2(n_\ell x_i)\\
    &+\sum_{\substack{(\ell_1,\ldots,\ell_{n-1},\ell_{n+1})\neq(\ell,\ldots,\ell)\\\tp{for } \ell=1,\dots, k}}n_{\ell_{n+1}}^{1-\alpha}\cos(n_{\ell_{n+1}}x_n)\prod_{i=1}^{n-1}\frac{n_{\ell_i}^{\frac{\alpha}{n}}}{(\ell_i+1)^{\frac{1}{n}}}\cos(n_{\ell_i}x_i)\cos(n_{\ell_{n+1}}x_i)\\
    =&\,\textrm{I}(x)+\textrm{II}(x).
\end{align*}
\textbf{Claim:} For $k$ sufficiently large, there exist constants $c,C>0$, independent of $k$, such that
\begin{align}
&-\int_{\Omega}\textrm{I}(x)u_n^{(k)}(x)\dif x\geq c \log k, \label{eq:claim1}\\
 &\Big|\int_{\Omega}\textrm{II}(x)u_n^{(k)}(x)\dif x\Big|\leq C.\label{eq:claim2}
\end{align}
Assuming the claim, we conclude the proof as
$$\int_{\Omega}\det(\D u^{(k)})\varphi^{(k)}\dif x=-\int_{\Omega}\big(\textrm{I}(x)+\textrm{II}(x)\big)u_n^{(k)}(x)\dif x\geq c\log k-C,$$
which tends to infinity as $k\to\infty$.

 To prove \eqref{eq:claim1}, we begin by expanding $\textrm{I}(x)u_n^{(k)}$:
\begin{align*}
    -\textrm{I}(x) u_{n}^{(k)}(x)=&\,\phi(x)\sum_{\ell=1}^k\frac{1}{\ell+1}n_\ell^{\alpha\frac{n-1}{n}+1-\alpha-1+\frac{\alpha}{n}}\prod_{i=1}^n\cos^2(n_\ell x_i)\\
    &+\phi(x)\sum_{\ell_n\neq \ell}\frac{n_\ell^{\alpha\frac{n-1}{n}+1-\alpha}n_{\ell_{n}}^{-1+\frac{\alpha}{n}}}{(\ell+1)^{\frac{n-1}{n}}(\ell_n+1)^{\frac{1}{n}}}\cos(n_{\ell_n}x_n)\cos(n_\ell x_n)\prod_{i=1}^{n-1}\cos^2(n_\ell x_i)\\
    =&\,\textrm{I}_1(x)+\textrm{I}_2(x).
\end{align*}
For $\textrm{I}_1(x)$, we note that the exponent of $n_\ell$ is zero and that $\textrm{I}_1\geq0$, and hence obtain
$$\int_{\Omega}\textrm{I}_1(x)\,\dif x\geq \int_{B_1(0)}\textrm{I}_1(x)\,\dif x\geq c\sum_{\ell=1}^k\frac{1}{\ell+1}\geq c \log k.$$
Considering now $\textrm{I}_2$, we note that when $n_{\ell_n}\neq n_\ell$, we have the estimate
$$\Big|\int_{\Omega}\phi(x)\cos(n_{\ell_n}x_n)\cos(n_\ell x_n)\dif x\Big|\leq C \frac{1}{|n_{\ell_n}-n_\ell|}.$$
Hence each term in the sum may be bounded as
\begin{align*}
    \Big|&\,\frac{n_\ell^{\frac{n-\alpha}{n}}n_{\ell_{n}}^{-\frac{n-\alpha}{n}}}{(\ell+1)^{\frac{n-1}{n}}(\ell_n+1)^{\frac{1}{n}}}\int_{\Omega}\phi(x)\cos(n_{\ell_n}x_n)\cos(n_\ell x_n)\prod_{i=1}^{n-1}\cos^2(n_\ell x_i)\dif x\Big|\\
    &\leq C \frac{n_\ell^{\frac{n-\alpha}{n}}}{n_{\ell_n}^{\frac{n-\alpha}{n}}}\frac{1}{|n_\ell-n_{\ell_n}|}\leq C\frac{1}{|n_\ell-n_{\ell_n}|^{\frac{\alpha}{n}}},
\end{align*}
where in the last step we have used that $|n_{\ell}-n_{\ell_n}|\geq \frac{1}{2}\max\{n_\ell,n_{\ell_n}\}$ by \eqref{ineq:nell}.
Thus, as there are $O(k^2)$ terms in the sum, we have the estimate
$$\textrm{I}_2\leq C k^2\max_{i\neq j}\frac{1}{|n_i-n_j|^{\frac{\alpha}{n}}}\leq C \text{ by \eqref{ineq:min_n_diff}}.$$
This proves the estimate \eqref{eq:claim1} in the claim.

To verify \eqref{eq:claim2}, we must consider the product of $\textrm{II}$ with the remaining factor:
\begin{align*}
    \textrm{II}(x) u_{n}^{(k)}(x)
    =\phi(x)\sum_{\ell_n=1}^k\sum_{\substack{(\ell_1,\ldots,\ell_{n-1},\ell_{n+1})\neq(\ell,\ldots,\ell)\\\tp{for } \ell=1,\dots, k}}&\frac{n_{\ell_{n+1}}^{1-\alpha}}{n_{\ell_n}^{\frac{n-\alpha}{n}}}\frac{1}{(\ell_n+1)^{\frac{1}{n}}}\cos(n_{\ell_n} x_n)\cos(n_{\ell_{n+1}}x_n)\\
    &\times\prod_{i=1}^{n-1}\frac{n_{\ell_i}^{\frac{\alpha}{n}}}{(\ell_i+1)^{\frac{1}{n}}}\cos(n_{\ell_i} x_i)\cos(n_{\ell_{n+1}}x_i).
\end{align*}
When $\ell_n=\ell_{n+1}$, the contributions to this sum are
\begin{align*}
    \frac{n_{\ell_{n}}^{\frac{\alpha}{n}-\alpha}}{(\ell_n+1)^{\frac{1}{n}}}\cos^2(n_{\ell_n} x_n)
    \prod_{i=1}^{n-1}\frac{n_{\ell_i}^{\frac{\alpha}{n}}}{(\ell_i+1)^{\frac{1}{n}}}\cos(n_{\ell_i} x_i)\cos(n_{\ell_n}x_i)\\
    =\frac{\cos^2(n_{\ell_n}x_n)}{(\ell_n+1)^{\frac{1}{n}}}\prod_{i=1}^{n-1}\frac{n_{\ell_i}^{\frac{\alpha}{n}}}{n_{\ell_n}^{\frac{\alpha}{n}}}\frac{1}{(\ell_i+1)^{\frac{1}{n}}}\cos(n_{\ell_i}x_i)\cos(n_{\ell_n}x_i).
\end{align*}
For each $i=1,\ldots,n-1$, if $\ell_i=\ell_n$, the factor in the product is bounded by $1$. If $\ell_i\neq\ell_n$ (note that there exists at least one such $i$), then
$$\Big|\int_{\Omega}\phi(x)\frac{\cos^2(n_{\ell_n}x_n)}{(\ell_n+1)^{\frac{1}{n}}}\frac{n_{\ell_i}^{\frac{\alpha}{n}}}{n_{\ell_n}^{\frac{\alpha}{n}}}\frac{\cos(n_{\ell_i}x_i)\cos(n_{\ell_n}x_i)}{(\ell_i+1)^{\frac{1}{n}}}\dif x\Big|\leq C \frac{n_{\ell_i}^{\frac{\alpha}{n}}}{n_{\ell_n}^{\frac{\alpha}{n}}}\frac{1}{|n_{\ell_i}-n_{\ell_n}|}\leq C \frac{1}{|n_{\ell_i}-n_{\ell_n}|^{\frac{n-\alpha}{n}}}.$$
Thus the total contribution from terms of this form is bounded by 
$$k^{n}\frac{1}{|n_{\ell_i}-n_{\ell_n}|^{\frac{n-\alpha}{n}}}\leq C k^n\max_{i\neq j}|n_{\ell_i}-n_{\ell_j}|^{-\frac{n-\alpha}{n}}\leq C\text{ by \eqref{ineq:min_n_diff}}.$$
Finally, if $\ell_{n}\neq \ell_{n+1}$, we have
\begin{align*}
    &\frac{n_{\ell_{n+1}}^{1-\alpha}}{n_{\ell_n}^{\frac{n-\alpha}{n}}}\frac{1}{(\ell_n+1)^{\frac{1}{n}}}\cos(n_{\ell_n} x_n)\cos(n_{\ell_{n+1}}x_n)\prod_{i=1}^{n-1}\frac{n_{\ell_i}^{\frac{\alpha}{n}}}{(\ell_i+1)^{\frac{1}{n}}}\cos(n_{\ell_i} x_i)\cos(n_{\ell_{n+1}}x_i)\\
    &= \frac{n_{\ell_{n+1}}^{\frac{n-\alpha}{n}}}{n_{\ell_n}^{\frac{n-\alpha}{n}}}\frac{1}{(\ell_n+1)^{\frac{1}{n}}}\cos(n_{\ell_n} x_n)\cos(n_{\ell_{n+1}}x_n)\prod_{i=1}^{n-1}\frac{1}{(\ell_i+1)^{\frac{1}{n}}}\frac{n_{\ell_i}^{\frac{\alpha}{n}}}{n_{\ell_{n+1}}^\frac{\alpha}{n}}\cos(n_{\ell_i} x_i)\cos(n_{\ell_{n+1}}x_i).
\end{align*}
As at least one $\ell_i\neq \ell_{n+1}$ and also $\ell_n\neq\ell_{n+1}$, we have an estimate (with this $i$) on the integral of
$$\frac{n_{\ell_{n+1}}^{\frac{n-\alpha}{n}}}{n_{\ell_n}^{\frac{n-\alpha}{n}}}\frac{1}{|n_{\ell_n}-n_{\ell_{n+1}}|}\frac{n_{\ell_i}^{\frac{\alpha}{n}}}{n_{\ell_n^\frac{\alpha}{n}}}\frac{1}{|n_{\ell_i}-n_{\ell_{n+1}}|}\leq C\max_{i\neq j}\frac{1}{|n_{\ell_i}-n_{\ell_j}|}. $$
Then, summing over all such terms, we have
$$\Big|\int_{\Omega}\textrm{II}(x)u_n^{(k)}(x)\dif x\Big|\leq C k^n \max_{i\neq j}|n_{\ell_i}-n_{\ell_j}|^{-\frac{n-\alpha}{n}}+Ck^{n+1}\max_{i\neq j}|n_{\ell_i}-n_{\ell_j}|^{-1}\leq C,$$
with $C$ independent of $k$ by \eqref{ineq:min_n_diff}.
This completes the proof of \eqref{eq:claim2}, and hence of the proposition.
\end{proof}

\appendix
\section{Sharp criteria for $\mathscr{H}^1$ bounds}\label{app:Orlicz}
As will become transparent from the proofs below, sharp criteria for the Hardy bound of Theorem \ref{thm:main} are only visible on the scale of Orlicz spaces. For references concerning Orlicz spaces and Young functions and for the definitions of the concepts used here, we refer to the monographs of Adams--Fournier \cite{Adams2003} and Rao--Ren \cite{Rao1991} as well as Iwaniec--Martin \cite{Iwaniec2001}. Below we give a short presentation of some concepts that are relevant for the main result of this section, Theorem~\ref{thm:hardy_boundOrlicz}.

An \textbf{Orlicz function} $\varphi\colon [0,\infty)\to [0,\infty)$ is a continuous, increasing function such that $\varphi(0)=0,\, \lim_{t\to \infty} \varphi(t)=\infty$.
We define the \textbf{Luxemburg functional} associated to $\varphi$ by 
$$\|f\|_{\lebe^\varphi(\Om)}\equiv \inf\left\{\la>0:\int_\Om \varphi\left(\frac{|f(x)|}{\la}\right)\dif x\leq 1\right\}.$$
The \textbf{Orlicz space} $\lebe^\varphi(\Om, \mathbb V)$ is the space of equivalence classes of $f\colon \Omega \to \mathbb{V}$ such that $\|f\|_{\lebe^\varphi(\Om)}<\infty$ and it is a complete metric space.  In general, $\Vert \cdot \Vert_{\lebe^\varphi(\Omega)}$ is a not norm, but when $\varphi$ is convex it is and $\lebe^\varphi(\Omega,\mathbb V)$ is a Banach space in that case; when $\varphi$ is convex, we say that it is a \textbf{Young function}.

An Orlicz function $\varphi$ is said to satisfy the \textbf{$\De_2$ condition} globally (respectively near infinity) 
if there exists $k>0$ such that for all $s\geq 0$ 
(respectively all $s\geq s_0$ for $s_0>0$ fixed),
$$\varphi(2s)\leq k\varphi(s).$$
If $\varphi$ satisfies the $\De_2$ condition globally we simply write $\varphi\in \De_2$.
This is equivalent to the existence of a constant $c>0$ such that for all $s\geq 0$ (respectively $s\geq s_0)$, $$\frac{1}{c}sa(s)\leq \varphi(s)\leq c sa(s).$$
If $\varphi_1$ and $\varphi_2$ are Orlicz functions, we say that \textbf{$\varphi_2$ dominates $\varphi_1$} 
globally (respectively near infinity) and write $\varphi_1\preceq\varphi_2$
if there exists $k>0$ such that for all $s\geq 0$ (respectively $s\geq s_0$),
$$\varphi_1(s)\leq \varphi_2(ks).$$
If $\varphi_1\preceq\varphi_2$ and $\varphi_2\npreceq\varphi_1$, we say that $\varphi_2$ \textbf{strictly dominates} $\varphi_1$ and write $\varphi_1\prec\varphi_2$.

For a given Young function $\varphi$, we define its \textbf{Young conjugate} $\varphi^*$ by the Legendre transform 
$$\varphi^*(t)\equiv \max_{s\geq 0}\{ts-\varphi(s)\}.$$
Thus for any $s,t\geq 0$, we have Young's inequality
$$st\leq \varphi(s)+\varphi^*(t).$$

\begin{thm}
Let $(\varphi,\varphi^* )$ be a pair of  Young conjugate functions such that $\varphi\in \Delta_2$ globally. Then the dual of $\lebe^{\varphi}(\Omega,\mathbb{V})$ is $\lebe^{\varphi^*}(\Omega, \mathbb V).$
\end{thm}

We will also use Orlicz--Sobolev spaces, defined similarly to the Zygmund--Sobolev spaces of Section~\ref{sec:zygmund}. In order to speak of distributional derivatives, however, we require functions in $\lebe^\varphi(\Omega,\mathbb V)$ to be locally integrable,  and so we now assume that $\varphi$ is an Orlicz function that dominates $t$ near infinity. 
Under this assumption, we define the \textbf{Orlicz--Sobolev space} $W^{k,\varphi}(\Omega,\mathbb V)$ as the space of those distributions $f \in \mathscr{D}'(\Omega,\mathbb V)$ such that, for all multi-indices $\alpha$ with $|\alpha|\leq k$, we have $\partial^\alpha f\in \lebe^\varphi(\Omega,\mathbb V)$.
We also define the negative Orlicz--Sobolev space ${W}^{-k,\varphi}(\R^n)$ as the space of those tempered distributions $f$ such that
$$\Vert f \Vert_{{W}^{-k,\varphi}(\R^n)}\equiv \left\Vert \mathcal F^{-1}\left(\frac{\mathcal F f(\xi)}{|\xi|^k}\right)\right\Vert_{\lebe^\varphi(\R^n)}<\infty.$$
We say that $f\in {W}^{-k,\varphi}_{\locc}(\R^n)$ if, for any test function $\phi$, we have $f\phi \in {W}^{-k,\varphi}(\R^n)$.

Let $\Om_1$, $\Om_2$ be open sets and let $T$ be a map from a linear subspace of the measurable functions on $\Om_1$ to the measurable functions on $\Om_2$. We say $T$ is quasilinear if there exists $C>0$ such that
$$|T(f+g)(x)|\leq C(|Tf(x)+|Tg(x)|),\quad |T(\la f)(x)|\leq |\la||Tf(x)|$$
for a.e. $x\in \Om_1$, all $f$ and $g$ in the domain of $T$, and all $\la\in \R$.

Given open sets $\Omega_1, \Omega_2\subset \R^n$, let $L\subset \lebe^0(\Omega)$ be a subspace of the space of measurable functions in $\Omega_1$ and consider a linear operator $T\colon L\to \lebe^0(\Omega_2)$. Given two Orlicz functions $\varphi$ and $\psi$, we say $T$ is of strong type $(\varphi,\psi)$ if there exists a constant $C>0$ such that
$$\|Tf\|_{\lebe^\psi(\Om_2)}\leq C\|f\|_{\lebe^\varphi(\Om_1)}$$
for all $f\in \lebe^\varphi(\Omega_1)$. We say that $T$ is of weak type $(\varphi,\psi)$ if there exists $C>0$ such that
$$|\{x\in\Om_2:Tf(x)>\la\}\leq 1/\psi\left(\frac{\la}{C\|f\|_{\lebe^\varphi(\Om_1)}}\right)$$
for all $f\in \lebe^\varphi(\Om_1)$ and $\la>0$. For our purposes, $T$ will be a Calder\'on--Zygmund operator arising from a multiplier:

\begin{thm}[{\cite[\S 12.12]{Iwaniec2001}}]\label{thm:HM_orlicz}
	    Let $\varphi$ be an Orlicz function for which there are numbers $1<p<q<\infty$ such that $t^{-p}\varphi(t)$ is increasing and $t^{-q}\varphi(t)$ is decreasing. 	Let $m$ be a zero-homogeneous H\"ormander--Mihlin multiplier, so $m$ corresponds to a Calder\'on--Zygmund operator $T_m$.
	
	Then $T_m$ is of weak type $(\varphi,\varphi)$ if and only if $\varphi$ is $\De_2$  and is strong type $(\varphi,\varphi)$ if and only if $\varphi,\varphi^*\in \De_2$.
\end{thm}

We can now state the main result of this section, providing sharp assumptions on the necessary and sufficient relationship between the integrability of the function $v$ and its constraint $\cala v$ in order to obtain the $\mathscr{H}^1$ bound:
\begin{thm}\label{thm:hardy_boundOrlicz}
Let $F\colon \mathbb{V}\to \R$ be $s$-homogeneous and $\cala$-quasiaffine, where $s\geq 2$. We suppose $\varphi$, $\psi$ are Young functions such that there exist $p,q\in(1,\infty)$ such that $t^{-p}\varphi$ and $t^{-p}\psi$ are increasing and $t^{-q}\varphi$ and $t^{-q}\psi$ are decreasing. Moreover, suppose $t^s\preceq \varphi\preceq t^s\log(1+t)$ and $$\big(\varphi\circ t^{\frac{1}{s-1}}\big)^*\circ\big(t\log(1+t)\big)\preceq \psi\preceq t^s\log^s(1+t),$$
where $*$ denotes Young conjugate function.
Then
\begin{align*}
	  \begin{rcases}
	  v\in L^\varphi_{\locc}(\R^n,\mathbb{V})\\
	  \cala v\in {\sobo}{_{\locc}^{-l,\psi} (\R^n,\mathbb{W})}
	  \end{rcases}
	  \implies
	  F(v)\in\mathscr{H}^1_{\locc}(\R^n),
	\end{align*} 
In fact, for any $R>0$ we have the estimate
\begin{equation*}
    \int_{B_R(0)}\mathcal M_{\locc}[F(v)](x)\dif x\leq C \|\eta v\|_{L^\varphi(B_{R+2}(0))} \|\cala (\eta v)\|_{{ \dot \sobo}^{-l,\psi}},
\end{equation*}
where $\eta\in C^\infty_c(B_{R+2}(0))$ is arbitrary.
\end{thm}
\begin{rmk}
Before proceeding to the proof, we note that if either $t^s\log^s(1+t)\preceq\psi$ or $t^s\log(1+t)\preceq\varphi$, then we already have either $L^\psi_{\locc}\hookrightarrow L^s\log^s L_{\locc}$ or $L^\varphi_{\locc}\hookrightarrow L^s\log L_{\locc}$, and hence may apply Theorem \ref{thm:main} directly.
\end{rmk}
\begin{proof} 
The proof follows almost the same lines as the proof of Theorem \ref{thm:hardy_bound}. We now extend the proof of that theorem to the general case for $\varphi$ and $\psi$. \\
First, we note that the assumptions $t^{-p}\varphi$ and $t^{-p}\psi$ are increasing and $t^{-q}\varphi$ and $t^{-q}\psi$ are decreasing are the assumptions necessary to obtain the H\"ormander-Mihlin interpolation Theorem~\ref{thm:HM_orlicz}.  Thus we may make the usual Helmholz-Hodge decomposition analogous to Proposition \ref{prop:HHHH} to obtain $u\in W^{k,\varphi}(\R^n;\mathbb{U})$, $w\in W^{l,\psi} (\R^n;\mathbb{W})$ such that
$$v =\mathcal{B}u+\mathcal{A}^*w.$$
Moreover,
$$\|\mathcal{B}u\|_{L^\varphi(\R^n)}\leq C\|v\|_{L^\varphi(\R^n)},\quad \|\mathcal{A}^*w\|_{L^\psi(\R^n)}\leq C\|\mathcal{A}v\|_{{\dot W}^{-l,\psi}(\R^n)}.$$
We now proceed as in the proof of Theorem \ref{thm:hardy_bound} in making the decomposition 
\begin{align}\label{eq:nonlin_diffOrlicz}
F(\B u+\cala^*w)-F(\B u)=\sum_{|\alpha|=s}\sum_{\beta<\alpha} c_{\alpha,\beta} (\B u)^{\beta}(\cala^*w)^{\alpha-\beta}.
\end{align} 
As $t^s\preceq \varphi\preceq t^s\log(1+t)$ and $$\big(\varphi\circ t^{\frac{1}{s-1}}\big)^*\circ\big(t\log(1+t)\big)\preceq \psi\preceq t^s\log^s(1+t),$$
 one obtains, via a simple calculation, that $\psi\succeq t^s\log(1+t)$. The term $|\beta|=0$ in \eqref{eq:nonlin_diffOrlicz} is therefore controlled by $\cala^* w$ in $L^s\log L$, hence in $L^\psi$. Dealing with the final term, $|\beta|=n-1$, we write $\tilde\varphi=\varphi\circ t^{\frac{1}{s-1}}$ and make the estimate
\begin{align*}
    \int_{B_{R+2}}|\B u|^{s-1}|\cala^*w|\log(1+|\cala^*w|)\dif x\leq&\,\||\mathcal{B} u|^{s-1}\|_{L^{\tilde\varphi}}\||\cala^* w|\log(1+|\cala^* w|)\|_{L^{\tilde{\varphi}^*}}\\
    \lesssim&\, \|\mathcal{B} u\|_{L^\varphi}\|\cala^* w\|_{L^\psi}
\end{align*}
via an obvious duality. This allows us then to conclude the proof as before, with analogous estimate.
\end{proof}

\begin{ex}
To show that these interpolated conditions on $\varphi$ and $\psi$ are sharp, we follow a similar procedure to that of Example \ref{ex:d}. Suppose that $t^s\preceq \varphi\prec t^s\log(1+t)$ and $\psi\prec \big(\varphi\circ t^{\frac{1}{s-1}}\big)^*(t\log(1+t))$, (as $s=2$, this means $\psi\prec \varphi^*\circ (t\log(1+ t))$). Note that as we are assuming $t^2\preceq\varphi$, we therefore have that the conjugate $\varphi^*$ satisfies $\varphi^*\preceq t^2\preceq \varphi$. Thus we have that $\psi$ is strictly dominated by $\varphi\circ(t\log(1+t))$. Without loss of generality, we therefore assume 
$$\varphi\preceq\psi\prec\varphi\circ(t\log(1+t)).$$
We will again construct our example as 
$$v=(f(x_1),0),\quad \tilde v=(\tilde f(x_1),0)$$
for suitable choices of $f$, $\tilde{f}$.\\
Choose $f:(-1,1)\to\R$ such that $f(x_1)>1$ for all $x_1\in(-1,1)$ and
$$f \in \lebe^\psi\hookrightarrow \lebe^\varphi,$$
satisfying that
$$f\varphi^{-1}(\psi(f))\log(1+f)\not\in \lebe^1.$$
This is possible as $\psi\prec \varphi\circ (t\log(1+t))$ implies directly that 
$$\big(\varphi^{-1}(\psi(t))\big)^2\prec t^2\log^2(1+t).$$
Now let $\tilde f(x_1)=g(|f|)f(x_1)$ where
$$g(t)t=\varphi^{-1}(\psi(t))\text{ for }t>0.$$
Then $\tilde v$, defined above, is in $L^\varphi$. Moreover, it is clear that 
$$\div v=\partial_{x_1}f \in W^{-1,\psi},\quad \curl \tilde v=0\in W^{-1,\psi}.$$
Taking the dot product, we obtain
	\begin{equation}
	    \begin{aligned}
	    v\cdot \tilde v\log(1+v\cdot \tilde v)=&\,f^2g(f)\log\big(1+f^2g(f)\big)\\
	    =&\,f\varphi^{-1}(\psi(f))\log\big(1+f\varphi^{-1}(\psi(f))\big)\\
	    \geq &\, c f\varphi^{-1}(\psi(f))\log(1+f),
	    \end{aligned}
	\end{equation}
	for some $c>0$, which is not in $\lebe^1$ by assumption. For the final inequality, we have used that as $f>1$, $\varphi^{-1}(\psi(f))\geq \varphi^{-1}(\psi(1))$ and then possibly adjusted the constant. 
\end{ex}

{\small
\bibliographystyle{acm}
\bibliography{/Users/antonialopes/Dropbox/Oxford/Bibtex/library.bib}
}

\end{document}